\UseRawInputEncoding
\documentclass[11pt]{amsart}
\usepackage[foot]{amsaddr}
\usepackage{hyperref}
\usepackage{amsmath,amsfonts,amssymb,amscd,amsthm,amsbsy,amsxtra,bbm,bm, epsf,calc,comment,appendix}
\usepackage{color}
\usepackage{datetime}
\usepackage{latexsym}
\usepackage[english]{babel}
\usepackage{enumerate}
\usepackage{graphicx}
\usepackage{epsfig}

\def\XXint#1#2#3{{\setbox0=\hbox{$#1{#2#3}{\int}$ }
\vcenter{\hbox{$#2#3$ }}\kern-.6\wd0}}

\def\II{{\rm I\kern-0.5exI}}
\def\III{{\rm I\kern-0.5exI\kern-0.5exI}}

\newcommand{\norm}[1]{\lVert #1 \rVert}

\newcommand{\RR}{\mathbb{R}}
\newcommand{\gi}{\frac{1}{\gamma}}

\newcommand{\ZZ}{\mathbb{Z}}

\DeclareMathOperator*{\argmin}{argmin}

\DeclareMathOperator{\loc}{\textup{loc}}

\DeclareSymbolFont{bbold}{U}{bbold}{m}{n}
\DeclareSymbolFontAlphabet{\mathbbold}{bbold}

\newcommand{\vp}{\varphi}

\newcommand{\spt}{\textup{spt}}

\numberwithin{equation}{section}
\newtheorem{theorem}{Theorem}[section]
\newtheorem{lemma}[theorem]{Lemma}
\newtheorem{prop}[theorem]{Proposition}

\theoremstyle{remark}
\newtheorem{remark}[theorem]{Remark}

\theoremstyle{definition}
\newtheorem{definition}[theorem]{Definition}
\author{Matt Jacobs}
\address{Department of Mathematics, 
Purdue University, 
Mathematical Sciences Bldg, 150 N University St, West Lafayette, IN 47907}
\email{jacob225@purdue.edu}
\title{Lagrangian solutions to the Porous Media Equation and reaction diffusion systems}
\begin{document}
\maketitle

\begin{abstract}
In this paper, we construct global-in-time forward and backward Lagrangian flow maps along the pressure gradient generated by weak solutions of the Porous Media Equation.  The main difficulty is that when the initial data has compact support, it is well-known that the pressure gradient is not a BV function.  Thus, the theory of regular Lagrangian flows cannot be applied to construct the flow maps.  To overcome this difficulty, we develop a new argument that combines Aronson-B\'enilan type estimates   with the quantitative Lagrangian flow theory of Crippa and De Lellis to show that certain doubly logarithmic quantities measuring the stability of flow maps do not blow up fast enough to prevent compactness.
Our arguments are sufficiently flexible to handle the Hele-Shaw limit and a multispecies generalization of the Porous Media Equation where the equation is replaced by a coupled hyperbolic-parabolic system of reaction diffusion equations.  As one application of our flow maps, we are able to construct solutions where different species cannot mix together if they were separated at initial time.

\end{abstract}

\section{Introduction}
The Porous Media Equation (PME) is a non-linear analogue of the heat equation that has various important physical applications \cite{vazquez_book}.  PME describes the evolution of a density $\rho$ according to Darcy's law, where the pressure function $p$ is coupled to the density through a convex function $e$.  If we also assume the presence of a growth/decay term $G$, the equation takes the form   
\begin{equation}\label{eq:pme_intro}
\partial_t \rho-\nabla \cdot (\rho\nabla p)=\rho G, \quad \rho p=e(\rho)+e^*(p),
\end{equation}
where $e^*$ is the convex conjugate of $e$.
Here we will focus on the classical coupling $p=\rho^{\gamma}$ for some parameter $\gamma\geq 1$ (corresponding to $e(\rho)=\frac{1}{\gamma+1}\rho^{\gamma+1}$), as well as the incompressible limit $\gamma\to\infty$, where the pressure-density relation degenerates to the implicit condition $p(1-\rho)=0$. 

Although PME is one of the most well-known examples of a non-linear parabolic PDE, to the best of our knowledge, the construction of Lagrangian flow maps along the pressure gradient has only been accomplished in very limited settings (in one dimension, in higher dimensions with radial symmetry, or when the pressure is known to have good regularity properties) \cite{shmarev_book, shmarev_vazquez, shmarev_seg_solutions}.
The main difficulty is that when the initial density $\rho^0$ is not bounded away from zero, the pressure gradient does not belong to $BV$ regardless of the smoothness of $\rho^0$. Indeed, one can only guarantee that $\Delta p$ is a singular measure \cite{vazquez_book}.  As a result, the pressure gradient has insufficient regularity to apply the theory of regular Lagrangian flows \cite{diperna_lions, ambrosio_bv}, despite the many recent advances and reformulations of the theory \cite{ambrosio_2008, crippa_delellis, jabin_flows, bouchut_crippa, nguyen_flow} to name just a few. 

Let us emphasize that in \cite{nguyen_flow}, the author constructs a vector field whose derivatives can be written as singular integrals of a Borel measure (i.e. essentially the same regularity that one expects for the pressure gradient in PME) where the regular Lagrangian flow theory fails and there is nonuniqueness of flow maps. Thus, to construct Lagrangian flows along the pressure gradient of PME, one cannot appeal to a general theory of flow maps, some particular properties of PME must be exploited.  

Unsurprisingly, the pressure gradient in PME has much more structure than an arbitrary vector field with poor regularity.  To understand this better, we can use the relation
 $p=\rho^{\gamma}$ to rewrite (\ref{eq:pme_intro}) in terms of $p$, leading to the  equivalent formulation
\begin{equation}\label{eq:pme_pressure}
    \partial_t p-|\nabla p|^2-\gamma p(\Delta p+G)=0.
\end{equation}
From (\ref{eq:pme_pressure}), one can see that PME is a degenerate parabolic equation, whose second order irregularities must occur in the vicinity of the level set $\{p=0\}$.  Indeed, it is known that quantities of the form $\int_{[0,T]\times\RR^d}p|D^2p|^2$  are finite under rather general assumptions on the structure of $G$
\cite{pqm, gpsg, perthame_david}. 
  Hence, there is hope to construct Lagrangian flows provided that one can show that most trajectories stay away from $\{p=0\}$.
  
  From a heuristic perspective, it is not so clear what should happen here.  Interpreting PME through the lens of the JKO scheme \cite{JKO1998variational, otto2001pme}, each particle attempts to move to a new location where the pressure is lower while not travelling too far.  On the other hand, the particles are carrying mass, thus, as they travel to regions of lower pressure, one expects that the pressure itself will rise.  Because of these competing phenomena, it unclear whether one should expect the pressure to increase or decrease along particle trajectories.

To gain a better understanding of what is happening, let us assume formally that $X$ is a Lagrangian flow map satisfying the flow equation $\partial_t X=-\nabla p\circ X$.   If we differentiate the pressure with respect to time along particle trajectories, we discover that 
\begin{equation}\label{eq:formal_trajectory_calculation}
\frac{d}{dt}\big( p\circ X)=\big( \partial_t p-|\nabla p|^2\big)\circ X=\big(\gamma p(\Delta p+G)\big)\circ X.
\end{equation}
Hence, the value of the pressure along trajectories is essentially controlled by the quantity $\gamma(\Delta p+G)$.  
In the most classical case, $G=0$, nearly uniform bounds on the negative part of $\gamma(\Delta p+G)$ are known through the celebrated  Aronson-Ben\'ilan estimates \cite{ab}, which provide the existence of a constant $C>0$ such that $\gamma\Delta p\geq -C/t$ for all $t>0$.
Unfortunately, for the version of PME that we care about the most (a multispecies generalization that will be described in detail shortly), a uniform lower bound on $\gamma(\Delta p+G)$ appears to be impossible \cite{gpsg, perthame_david} (even if it is allowed to blow up at time zero).  
Without uniform bounds, we cannot rule out the possibility that \emph{every} trajectory spends some time near $\{p=0\}$ i.e. there may be no ``good'' trajectories where the regular Lagrangian flow theory applies.  Thus, we cannot just hope to use existing theory.

To overcome the aforementioned difficulties, we develop a new argument based on the quantitative estimates in \cite{crippa_delellis} that only requires us to control $\gamma(\Delta p+G)$ in a weighted $L^1$ space.  While the weighted $L^1$ bound cannot guarantee that trajectories stay away from $\{p=0\}$, together with equation (\ref{eq:pme_pressure}) it implies the following logarithmic version of Gronwall's inequality
\begin{equation}
    \int_{\RR^d} \rho^0(x)\log(1+\frac{1}{p(t,X(t,x))})\lesssim_t \int_{\RR^d} \rho^0(x)\log(1+\frac{1}{p(0,X(0,x))}),
\end{equation}
which provides some control on how trajectories interact with the zero level set.
Although the logarithmic bound appears to be extremely weak, it is highly compatible with the estimates in \cite{crippa_delellis} which also consider logarithmic quantities related to the flow map.  Ultimately, by combining these ideas we are able to construct the forward and backward Lagrangian flow maps along $-\nabla p$.  Notably, our arguments are sufficiently flexible to handle the Hele-Shaw limit (i.e. $\gamma\to\infty$) and a multispecies generalization of the Porous Media Equation where the equation is replaced by a coupled hyperbolic-parabolic system of reaction diffusion equations.  As one application of our flow maps, we are able to construct solutions where different species cannot mix together if they were separated at initial time.

In the rest of the introduction, we will introduce the multispecies generalization of PME, describe our main results, and finally lay out a road map for the rest of the paper.

\subsection{Multispecies generalizations of PME}
A particularly important application of PME is the modelling of living cells and tissues, particularly in the context of tumor growth \cite{BYRNE2003567, Preziosi2008, Ranft20863, pqv}.  When cells proliferate and grow, there is a buildup of mechanical pressure, which both pushes cells down the pressure gradient and affects growth rates via the biological phenomenon of contact inhibition \cite{pqv}.  This can be modeled by PME with a pressure dependent source term. 

For realistic modelling, it is important to take into account multiple cell populations (e.g. healthy cells versus tumor cells) and nutrient availability.   In this paper, we will be interested in studying a system of evolution equations for a finite number of cell populations with densities $\rho_1, \ldots, \rho_{\ell}$, whose total density $\rho=\sum_{i=1}^{\ell} \rho_i$ evolves according to PME with a source term. Each individual population evolves according to the continuity equation
\begin{equation}\label{eq:multi_pme}
    \partial_t \rho_i-\nabla \cdot (\rho_i\nabla p)=\rho_i G_i, \rho p=e(\rho)+e^*(p)
\end{equation}
where once again $e(\rho)=\frac{1}{\gamma+1}\rho^{\gamma+1}$ for some $\gamma>1$ or the incompressible limit $\gamma\to\infty$,
 $G_i$ is a growth function that depends on the pressure and a nutrient variable $n$. The nutrient is coupled to the other variables through the diffusion equation
\begin{equation}\label{eq:nutrients}
    \partial_t n-\alpha\Delta n=-n\sum_{i=1}^{\ell} \beta_i\rho_i,
\end{equation}
where $\alpha, \beta_i$ are scalars that determine the diffusion rate and consumption rate of the nutrients respectively.  The connection between the system (\ref{eq:multi_pme}-\ref{eq:nutrients}) and the classical PME (\ref{eq:pme_intro}) can be seen by
summing (\ref{eq:multi_pme}) over each of the populations.  Doing so recovers the equation
\begin{equation}\label{eq:pme}
\partial_t \rho-\nabla \cdot (\rho\nabla p)=\rho G, \quad p=\rho^{\gamma},
\end{equation}
where $G=\sum_{i=1}^{\ell} \frac{\rho_i}{\rho}G_i$.  

In the case of multiple cell populations, the model is a challenging system of coupled PDEs.  Indeed, existence of solutions to these systems in dimensions $d>1$ was only achieved recently (see \cite{hilhorst, carrillo_1d} for results in one dimension) in the series of papers \cite{gpsg, price_xu, liu_xu, jacobs_2021} for $d>1$, while well-posedness remains open. The difficulty of these systems stems from the fact that although equation (\ref{eq:pme}) is degenerate parabolic, the evolution equations for the individual populations (\ref{eq:multi_pme}) are hyperbolic.  Hence, the equation does not have any regularizing effect on the $\rho_i$ (for instance discontinuities at initial time will persist throughout the evolution).  While \cite{carrillo_1d} was able to obtain strong compactness of the $\rho_i$ in one dimension, the situation in $d>1$ is more complicated. Following the approach of \cite{gpsg}, all of the results for $d>1$ have constructed solutions by obtaining strong compactness for the pressure variable instead. The advantage of working with the pressure is that one can focus on the good properties of equation (\ref{eq:pme}), however, this approach cannot say much about the properties of the limiting $\rho_i$. 
In the rest of this paper, we will focus on constructing Lagrangian solutions for the multispecies system (\ref{eq:multi_pme}-\ref{eq:nutrients}).  Since these equations are strictly more general than the classical PME, this will automatically provide Lagrangian solutions to (\ref{eq:pme_intro}) as well. 

In addition to the construction of Lagrangian solutions, we also answer an important open question 
about the possibility of population mixing in the multispecies model. Indeed, until now, it has been unknown whether the individual populations will remain  unmixed throughout the evolution if they were separated at initial time.  More concretely, given initial data $\{\rho_i^0\}_{i\in \{1,\ldots, \ell\}}$ such that $\min(\rho_i^0,\rho_j^0)=0$ for all $i\neq j$ one wishes to know whether it is possible to construct solutions such that $\min(\rho_i,\rho_j)=0$ almost everywhere for all $i\neq j$.   This question has been difficult to answer, as previous methods for constructing solutions to the system (\ref{eq:multi_pme}-\ref{eq:nutrients}) have not been able to obtain strong compactness for the individual populations $\rho_{i}$ along approximate sequences \cite{gpsg, price_xu, liu_xu, jacobs_2021}.  As a result, any nonconvex properties of the $\rho_i$ along the sequence are lost in the limit. Using our Lagrangian flow maps, we are able to answer this question in the affirmative by exploiting an explicit representation formula for the individual densities in terms of the flow map.
Hence, in addition to the construction of Lagrangian solutions for PME, the other main result of this paper is a conclusive answer to the mixing question. 

In the rest of the introduction, we will give a more precise mathematical description of our setup and main results.

 \subsection{Preliminaries and main results}\label{ssec:main_results}

We begin by giving a more concrete description of the growth terms and our important assumptions on them.  Throughout the paper we shall place the following assumptions on the $G_i$.
 \begin{enumerate}
     \item[(G1)] \label{GA:cts}Each $G_i:=G_i(p,n)$ is a continuous and uniformly bounded function of the pressure $p$ and nutrient $n$.
     \item[(G2)] \label{GA:hp} If the pressure is sufficiently high, no growth occurs regardless of nutrient availability, i.e. there exists some $p_h>0$ such that $G_i(p,n)< 0$ for all $i\in \{1,\ldots, \ell\},n\in [0,\infty)$ and $p> p_h$ (the value $p_h$ has been called the homeostatic pressure in the literature \cite{pqv}). 
     \item[(G3)]\label{GA:wd} The following mild technical condition on the derivatives holds: \[\max_{i\in \{1,\ldots, \ell\}}\Big( p\partial_p G_i(p,n)\Big)_++ |\partial_n G_i(p,n)|\in L^{\infty}_{\loc}([0,\infty)^2).\]
 \end{enumerate}
 When we pass to the incompressible limit $\gamma\to\infty$ we will require the $G_i$ to satisfy an additional positivity condition:
 \begin{enumerate}
     \item[(G4)]\label{GA:sd} $\min_{i\in \{1,\ldots, \ell\}}\inf_{(p,n)\in [0,\infty)^2} \frac{1}{2}G_i(p,n)-p\partial_p G_i(p,n)>0$.
 \end{enumerate}
 
  Unlike (G1-G3), this last condition is much more  restrictive from a modelling perspective.  (G4) forces $G_i(0,0)>0$, meaning the cells must grow even in the absence of nutrients. Let us note however that (G4) is not purely technical, some property related to (G4) is necessary to guarantee the nonmixing property in the incompressible case. Given two populations with growth functions satisfying $G_1(0,0)>0$ and $G_2(0,0)<0$, it is easy to cook up a scenario where population 1 instantaneously mixes into population 2.  For instance, this will always happen in a scenario where the initial nutrient value is everywhere zero and the populations are placed in starting condition where they separated, share a codimension 1 boundary, and both saturate the incompressible constraint on their respective supports.
 
 For the initial data, all of our conditions are on the total density $\rho^0=\sum_{i=1}^{\ell}\rho_i^0$, the corresponding initial pressure $p^0$, and the starting nutrient level $n^0$.  We shall require the following regularity conditions.
 \begin{enumerate}
     \item[(ID1)] $\rho^0\in L^1(\RR^d)$, $\rho^0\in [0, p_h^{\gi}]$, and $|x|^2\rho^0\in L^1(\RR^d)$.
     \item[(ID2)] $\nabla p^0, \nabla n^0\in L^2(\RR^d)$,  and $n^0\in L^{\infty}(\RR^d)$.
     \item[(ID3)] $\gamma\rho^0(\Delta p^0+\sum_{i=1}^{\ell} \frac{\rho_i^0}{\rho_i}G_i(p^0,n^0))_{-}^2\in L^1(\RR^d)$.
 \end{enumerate}
 When we pass to the incompressible limit we shall require the two following additional conditions 
  \begin{enumerate}
     \item[(ID4)] $\rho^0\in \{0,1\}$ almost everywhere.
     \item[(ID5)] There exists a constant $\lambda>0$ such that $\rho^0\log(1+\frac{1}{p^0})^{\lambda}\in L^1(\RR^d)$. 
     \end{enumerate}
 
 Next, we give a concrete description of the solutions that we are interested in constructing.
 \begin{definition}\label{def:complete_lagrangian_solution}
 We will say that a tuple $(\rho_1,\ldots, \rho_k, p, n)$ is a \emph{complete Lagrangian solution} to the system (\ref{eq:multi_pme}-\ref{eq:nutrients}) with initial data $(\rho_1^0,\ldots, \rho_{\ell}^0, n^0)$ if the following conditions are met.
 \begin{enumerate}[(i)] \label{def:complete_lagrangian_solution_1}
     \item $(\rho_1,\ldots, \rho_k, p,n)$ is a weak solution to  (\ref{eq:multi_pme}-\ref{eq:nutrients}) with initial data $(\rho_1^0,\ldots, \rho_{\ell}^0, n^0)$ such that for any $T>0$, $\rho=\sum_{i=1}^{\ell}\rho_i\in L^{\infty}([0,T];L^1(\RR^d)\cap L^{\infty}(\RR^d))$.
     \item \label{def:complete_lagrangian_solution_2} For all $t, s\geq 0$ there exist unique forward and backward flow maps $X, Y$ satisfying the Lagrangian Flow equations 
     \begin{equation}\label{eq:x_flow}
     X(t,s,x)=x-\int_{s}^{t+s} \nabla p(\tau, X(\tau, s, x))\, d\tau \quad \textup{for almost all} \; x\in \spt(\rho(s,\cdot)),
 \end{equation}
 and 
\begin{equation}\label{eq:y_flow}
     Y(t,s,x)=x+\int_{\min(s-t,0)}^s \nabla p(\tau, Y(\tau, s, x))d\tau \quad \textup{for almost all} \; x\in \spt(\rho(s,\cdot)).
 \end{equation}
 \item \label{def:complete_lagrangian_solution_3} There exists a constant $B\geq 0$ such that for all $t,s\geq 0$
 \begin{equation}\label{eq:x_push}
   e^{-tB}\rho(s+t,\cdot)  \leq X(t,s,\cdot)_{\#}\rho(s,\cdot)\leq e^{tB}\rho(s+t,\cdot)\quad \textup{for almost all } \; x\in \spt(\rho(s+t,\cdot)), 
 \end{equation}
 and
 \begin{equation}\label{eq:y_push}
   e^{-tB}\rho(s,\cdot)  \leq Y(t,s,\cdot)_{\#}\rho(s+t,\cdot)\leq e^{tB}\rho(s,\cdot)\quad \textup{for almost all } \; x\in \spt(\rho(s,\cdot)).
 \end{equation}
    \item \label{def:complete_lagrangian_solution_4}The maps satisfy the semigroup property
 \begin{align}\label{eq:semigroup}
     X(t,s,x)=X(t-t',s+t',X(t',s,x))\quad \textup{for almost all } \; x\in \spt(\rho(s,\cdot)), \\
     Y(t,s,x)=Y(t-t',s-t',Y(t',s,x)) \quad \textup{for almost all } \; x\in \spt(\rho(s,\cdot)),
 \end{align}
 and the
 inversion formulas \begin{align}\label{eq:map_inversion}
     X(t,s,Y(t,s+t,x))=x\quad \textup{for almost all }\; x\in \spt(\rho(t+s,\cdot)),\\
     Y(t,s,X(t,s-t,x'))=x'\quad \textup{for almost all} \; x'\in \spt(\rho(s,\cdot)).
 \end{align}
 \item \label{def:complete_lagrangian_solution_5} For any test function $\vp$, each $\rho_i$ satisfies the representation formula
 \begin{equation}\label{eq:rho_representation}
     \int_{\RR^d} \rho_i(s+t,x)\vp(x)=\int_{\RR^d} \rho_i(s,x)\vp(X(t,s,x))\exp\Big(\int_{s}^{t+s} G_i\circ X(\tau,s,x)d\tau\Big)\, dx,
 \end{equation}
 where $G_i\circ X(\tau,s,x)$ is shorthand for $G_i\Big(p\big(\tau, X(\tau,s,x)\big), n\big(\tau, X(\tau,s,x)\big)\Big)$.
 
 \end{enumerate}
 \end{definition}
 \begin{remark}
 The uniqueness of the flow maps along $-\nabla p$ guarantees the uniqueness of the $\rho_i$ and $n$ when $p$ is held fixed. However, we are not able to prove that the system itself has a unique solution. Indeed, we cannot rule out the possibility that there could be solutions with different pressure variables starting from the same initial data, and there are results that these types of systems do not have unique solutions \cite{shmarev_seg_solutions}.
 \end{remark}

 We are now ready to give our main results.  For convenience we restrict our attention to values of $\gamma\geq 1$. 

 \begin{theorem}\label{thm:main1}
 Given growth terms satisfying assumptions (G1-G3), initial data $(\rho_1^0, \ldots, \rho_{\ell}^0, n^0)$ satisfying (ID1-ID3),  and $\gamma\in [1,\infty)$, there exists a complete Lagrangian solution $(\rho_1,\ldots, \rho_{\ell}, n,p)$ for (\ref{eq:multi_pme}-\ref{eq:nutrients}). Furthermore if for some $i\neq j$ we have $\min(\rho_i^0(x),\rho_j^0(x))=0$ almost everywhere, then for every $t\geq 0$ we have $\min(\rho_i(t,x),\rho_j(t,x))=0$ almost everywhere in $x$.
 \end{theorem}
 
 \begin{theorem}\label{thm:main2}
Given growth terms satisfying assumptions (G1-G4) and initial data $(\rho_1^0, \ldots, \rho_{\ell}^0, n^0)$ satisfying (ID1-ID5) along with the additional condition $\rho^0\in \{0,1\}$ almost everywhere,  there exists a complete Lagrangian solution $(\rho_1,\ldots, \rho_{\ell}, n,p)$ to (\ref{eq:multi_pme}-\ref{eq:nutrients}) with $\gamma=\infty$, i.e. the incompressible system
\begin{equation}\label{eq:incompressible}
    \partial_t \rho_i-\nabla \cdot (\rho_i\nabla p)=\rho_i G_i, \quad p(1-\rho)=0, \quad \rho\leq 1
\end{equation}
\begin{equation}\label{eq:incompressible_nutrients}
    \partial_t n-\alpha\Delta n=-n\sum_{i=1}^{\ell}\beta_i\rho_i.
\end{equation}
Furthermore if for some $i\neq j$ we have $\min(\rho_i^0(x),\rho_j^0(x))=0$ almost everywhere, then for every $t\geq 0$ we have $\min(\rho_i(t,x),\rho_j(t,x))=0$ almost everywhere in $x$.
 \end{theorem}

 The rest of the paper is structured as follows.  In Section \ref{sec:estimates},
 we assume that we have a smooth solution to the system and collect a number of important estimates, most crucially, the weighted AB type estimate on $\gamma^2(\Delta p+G)_-^2$ and the weighted $L^1$ estimate on $\gamma|\Delta p+G|$.  In section \ref{sec:flows}, we show how these estimates can be used to establish stability properties for the associated Lagrangian flow maps.  In the final Section, Section \ref{sec:compactness}, we show how one can construct smooth approximations to the system and then take limits to prove the main results.

\section*{Acknowledgements}
 The author is greatful to Inwon Kim and Simone di Marino for some motivating conversations.  Some of the ideas in this paper began when the author was a Simons Fellow at the Simons Institute for Computing.

 \section{Estimates}\label{sec:estimates}
 
 Throughout this section, we will assume that we have a smooth solution $(\rho_1,\ldots, \rho_{\ell}, p, n)$ to the system (\ref{eq:multi_pme}-\ref{eq:nutrients}) where the initial data satisfies assumptions (ID1-ID3), the growth terms satisfy (G1-G3) and $\gamma\in [1,\infty)$. This will allow us to investigate properties of the system without having to worry about integrability or differentiability issues.  Our main goal in this section will be to build towards bounds on $p|D^2p|^2$ and  $\rho(\Delta p+G)_-^2$ that only depend on the information (ID1-ID4).  For notational convenience we shall use $Q_T$ to denote the space time domain $Q_T:=[0,T]\times\RR^d$ for any $T>0$.
 
 Our analysis and estimates will be focused on the ``nice'' parabolic equations (\ref{eq:pme}) and (\ref{eq:pme_pressure}), rather than the hyperbolic equation (\ref{eq:multi_pme}).  Nonetheless, we will still need to work with the individual densities $\rho_i$ through their presence in the growth term $G=\sum_{i=1}^{\ell}\frac{\rho_i}{\rho} G_i$.  A formal calculation shows that the ratios $c_i:=\frac{\rho_i}{\rho}$ satisfy the transport equation
\begin{equation}\label{eq:c_i}
    \partial_t c_i-\nabla c_i\cdot \nabla p=c_i(G_i-G).
\end{equation}
Since we have already assumed we are working with smooth solutions, we can use this formula without issue.  It will play an important role in some of the subsequent estimates.

  We begin with some standard estimates for PME type equations.
\begin{lemma}
Let $B=\sup_{(p,n)\in [0,\infty)^2}\max_i |G_i(p,n)|$
and fix some time $T\geq 0$.  For any $m\geq 1$
\begin{equation}\label{eq:edr_m}
\int_{\RR^d\times \{T\}} \frac{1}{m}\rho^m +\int_{Q_T} (m-1)\rho^{m-1} \nabla \rho\cdot \nabla p=e^{BmT}\int_{\RR^d}\frac{1}{m}(\rho^0)^m,
\end{equation}
and
\begin{equation}
\int_{\RR^d\times \{T\}} |x|^2\rho\leq e^{T(B+1)}\Big(\int_{\RR^d} |x|^2\rho^0+ \int_{Q_T} \rho|\nabla p|^2\Big).
\end{equation}
Furthermore,
$p\leq p_h$ almost everywhere.
\end{lemma}
\begin{proof}
The first relation follows from integrating equation (\ref{eq:pme}) against $\rho^{m-1}$ and using Gronwall's inequality. The second bound follows from integrating (\ref{eq:pme}) against $|x|^2$, using Young's inequality (one can first integrate against $|x|^2e^{-\delta|x|^2}$ and then send $\delta\to 0$ to check that $|x|^2$ is a valid test function), and then using Gronwall's inequality. 

For the bound $p\leq p_h$, we can multiply equation (\ref{eq:pme_pressure}) against $(p-p_h)_+$ to obtain
\[
\int_{\RR^d} \frac{d}{dt} \frac{1}{2}(p-p_h)_+^2-(p-p_h)_+|\nabla p|^2\leq \int_{\RR^d} \gamma p(p-p_h)_+\Delta p
\]
where we have used the fact that each $G_i\leq 0$ whenever $p\geq p_h$. After integrating by parts and dropping a good term, we see that
\[
\int_{\RR^d} \frac{d}{dt} \frac{1}{2}(p-p_h)_+^2+(\gamma-1)(p-p_h)_+|\nabla p|^2\leq 0,
\]
and the result follows.

\end{proof}

The following $L^2$ bounds will be important for our AB estimates.
\begin{lemma}\label{lem:p_grad_bound}
Let $B=\sup_{(p,n)\in [0,\infty)^2}\max_i |G_i(p,n)|$
and fix some time $T\geq 0$.  For any $\gamma\in [1,\infty)$
\begin{equation}\label{eq:edr}
\int_{\RR^d\times \{T\}} \rho\log(\rho) +\int_{Q_T} \rho\frac{|\nabla p|^2}{\gamma p}\leq \int_{\RR^d\times \{T\}}\rho+\int_{Q_T} B\rho\log(\rho).
\end{equation}
Furthermore,
if $\zeta:[0,\infty)\to\RR$ is a nonnegative increasing function such that $\zeta$ is $C^1$ on $(0,\infty)$, then 
\begin{equation}\label{eq:zeta_grad_p}
\int_{Q_T} \rho\zeta'(p)|\nabla p|^2\leq \zeta(p_h)\big(\norm{\rho^0}_{L^1(\RR^d)} +B\norm{\rho}_{L^1(Q_T)}\big).
\end{equation}
\end{lemma}
\begin{remark}
Note that the negative part of $\rho\log(\rho)$ can be controlled by the second moments of $\rho$. Indeed, one has $-\rho\log(\rho)\lesssim_d \rho^{1-\frac{1}{d+1}}\lesssim_d \rho(1+|x|^2)+(1+|x|^2)^{-d}.$
\end{remark}

\begin{proof}
 The first inequality follows from integrating (\ref{eq:pme}) against $\log(\rho+\delta)$,  and then sending $\delta\to 0$. 
 
For the second inequality, we integrate (\ref{eq:pme}) against $\zeta(p)$, to obtain
\[
\int_{\RR^d} \zeta(p)\partial_t \rho+\rho\zeta'(p)|\nabla p|^2=\int_{\RR^d} \rho \zeta(p)G.
\]
Note that $\zeta(p)\partial_t \rho=\frac{\zeta(p)}{\gamma p^{1-\gamma}}\partial_t p$.  If we set $\bar{\zeta}(a)=\int_0^a \frac{\zeta(a)}{\gamma a^{1-\gamma}}\, da$, then it follows that 
\[
\int_{\RR^d} \rho\zeta'(p)|\nabla p|^2+\frac{d}{dt} (\bar{\zeta}(p))=\int_{\RR^d} \rho \zeta(p)G.
\]
Integrating with respect to time, we get 
\[
\int_{\RR^d\times\{T\}} \bar{\zeta}(p)+\int_{Q_T} \rho\zeta'(p)|\nabla p|^2=\int_{\RR^d}\bar{\zeta}(p^0)+\int_{Q_T} \rho \zeta(p)G.
\]
Since $\zeta$ is positive and increasing, we have $\zeta(p)\leq \zeta(p_h)$ and $\bar{\zeta}(p)\leq \zeta(p_h)p^{\gi}=\zeta(p_h)\rho$.
The result now follows.
\end{proof}

For the nutrient equation, we have the following estimates that are standard for the heat equation
\begin{lemma}\label{lem:nutrient}
For any time $T>0$,
\begin{equation}
   \norm{n}_{W^{1,\infty}(\{T\}\times\RR^d)}\lesssim \norm{n^0}_{W^{1,\infty}(\{T\}\times\RR^d)}+\norm{\rho}_{L^{\infty}(Q_T)}\norm{n}_{L^{\infty}(Q_T)}T^{1/2}
\end{equation}
\begin{equation}
 \norm{\nabla n}_{L^2(\{T\}\times\RR^d)}^2+\norm{\partial_t n}_{L^2(Q_T)}^2\lesssim \norm{\nabla n^0}_{L^2(\RR^d)}^2+\norm{n}_{L^{\infty}(Q_T)}^2\norm{\rho}_{L^2(Q_T)}^2
\end{equation}
\end{lemma}

The next estimate is essentially taken directly from \cite{perthame_david}.  We reproduce the argument here since we are in the case of multiple populations, however, the differences are relatively minor.  For notational convenience we will adopt the shorthand
\begin{equation}
    u:=-\gamma (\Delta p+G).
\end{equation}
With this shorthand, equation (\ref{eq:pme_pressure}) now reads
\begin{equation}\label{eq:u_pme}
\partial_t p -|\nabla p|^2+ p u=0.
\end{equation}

\begin{prop}\label{prop:phessian}
For any time $T>0$ there exists a constant $C(T)$ such that
\begin{equation}
   \int_{\RR^d\times \{T\}}|\nabla p|^2+\int_{Q_T} p|D^2p|^2\leq C(T).
\end{equation}
For any increasing function $\eta:\RR\to\RR$ and $m\in [2,4]$
\begin{equation}\label{eq:grad_p_power_estimate}
  \int_{Q_T} \eta'(p) |\nabla p|^m\lesssim 1+\norm{\frac{\eta(p)^2}{p\eta'(p)^{\frac{2m-4}{4}}}}_{L^{\frac{m}{4-m}}(Q_T)}\big(\int_{Q_T} p|\Delta p|^2+p|D^2p|^2\big).
\end{equation}

\end{prop}
\begin{proof}
Integrating equation (\ref{eq:u_pme}) against $\gi u$, we get 
\[
\gi \int_{Q_T} u\partial_t p- u|\nabla p|^2+ pu^2=0.
\]
Note that $-\gi u|\nabla p|^2=G|\nabla p|^2+\Delta p|\nabla p|^2.$ 
Integrating by parts, we see that
\[
\int_{Q_T} \Delta p|\nabla p|^2=\int_{Q_T} 2p|D^2p|^2+2p\nabla \Delta p\cdot \nabla p=\int_{Q_T} 2p|D^2p|^2-2\Delta p|\nabla p|^2-2p|\Delta p|^2.
\]
Hence, we obtain the identity, 
\[
\int_{Q_T} \Delta p|\nabla p|^2=\frac{2}{3} \int_{Q_T} p|D^2p|^2-p|\Delta p|^2.
\]
Expanding $p|\Delta p|^2=\frac{1}{\gamma^2}u^2-\frac{2}{\gamma} puG+pG^2$, our combined work gives us
\[
\gi \int_{Q_T} u\partial_t p+\gamma G|\nabla p|^2+(1-\frac{2}{3\gamma}) pu^2-\frac{4}{3}puG+\frac{2\gamma}{3}pG^2+\frac{2\gamma}{3}p|D^2p|^2\leq 0.
\]
Applying Young's inequality to $p uG$, we can conclude that for any $\gamma\geq 1$
\begin{equation}\label{eq:d2p_1}
\int_{Q_T} \gi u\partial_t p+\gi pu^2+p|D^2p|^2\lesssim \int_{Q_T} pG^2+G|\nabla p|^2
\end{equation}

Now we turn our attention to the time derivative term.  We see that 
\[
\int_{Q_T} \gi u\partial_t p=\int_{Q_T} -G\partial_t p-\Delta p\partial_t p= \norm{\nabla p(T,\cdot)}_{L^2(\RR^d)}^2-\norm{\nabla p^0}_{L^2(\RR^d)}^2-\int_{Q_T} G\partial_t p.
\]
Recall that $G=\sum_{i=1}^{\ell} c_iG_i(p, n)$
where $c_i$ satisfy (\ref{eq:c_i}).  We then see that
\[
G\partial_t p=\sum_{i=1}^{\ell} c_i \big( \frac{d}{dt} \bar{G}_i(p,n)-\partial_n \bar{G}_i(p,n)\partial_t n\big),
\]
where $\bar{G}_i:\RR^2\to\RR$ is defined as $\bar{G}_i(p,n):=\int_0^p G_i(a,n)\, da$.
Hence,
\[
\int_{Q_T} G\partial_t p=\sum_{i=1}^{\ell}\int_{Q_T} \frac{d}{dt}\big(c_i\bar{G}_i(p,n)\big)-c_i\partial_n\bar{G}_i(p,n)\partial_t n-\bar{G}_i(p,n)\partial_t c_i.
\]
Using (\ref{eq:c_i}), we can now estimate
\begin{multline}\label{eq:d2p_2}
\Big|\int_{Q_T} G\partial_t p\Big|\leq \sum_{i=1}^{\ell} \norm{G_i}_{L^{\infty}(\RR^2)}\norm{p^0+p(T,\cdot)}_{L^1(\RR^d)}+\norm{\partial_n G_i}_{L^{\infty}(\RR^2)}\norm{p}_{L^2(Q_T)}\norm{\partial_t n}_{L^2(Q_T)}\\
+\sum_{i=1}^{\ell}\Big|\int_{Q_T} \bar{G}_i(p,n)\big(\nabla c_i\cdot \nabla p+c_iG_i-c_i\sum_jc_jG_j\big)\Big|.
\end{multline}
For the final integral we want to remove derivatives from $c_i$. Integrating by parts and then using Young's inequality, we get
\[
\int_{Q_T} \bar{G}_i(p,n)\nabla c_i\cdot \nabla p=\int_{Q_T} c_i (\gi u+G)\bar{G}_i(p,n)- c_iG_i(p,n)|\nabla p|^2
\]
\[
\leq \int_{Q_T} c_i \gi pu^2+\gi\frac{\bar{G}_i(p,n)^2}{p}+G\bar{G}_i(p,n)- c_iG_i(p,n)|\nabla p|^2.
\]
The first result now follows from combining the previous line with (\ref{eq:d2p_1}) and (\ref{eq:d2p_2}).

For the second result, we integrate by parts and then use Young's inequality to get
\[
\int_{Q_T} \eta'(p) |\nabla p|^m=-\int_{Q_T}\eta(p)(\Delta p|\nabla p|^{m-2}+(m-2)|\nabla p|^{m-4} D^2p:\nabla p\otimes \nabla p)
\]
\[
\leq \frac{1}{2}\int_{Q_T} a p|\Delta p|^2+a(m-2)p|D^2p|^2+a^{-1}(m-1)\frac{\eta(p)^2}{p}|\nabla p|^{2m-4}
\]
for some constant $a$.
After applying Holder's inequality we obtain
\[
\int_{Q_T} \eta'(p) |\nabla p|^m\lesssim a^{-1}\norm{\frac{\eta(p)^2}{p\eta'(p)^{\frac{2m-4}{4}}}}_{L^{\frac{m}{4-m}}(Q_T)}\norm{\eta'(p)|\nabla p|^m}_{L^1(Q_T)}^{\frac{2m-4}{m}}+ a\big(\int_{Q_T} p|\Delta p|^2+p|D^2p|^2\big).
\]
Since $\frac{2m-4}{m}\leq 1$ for $m\in [2,4]$, it follows that
\[
\int_{Q_T} \eta'(p) |\nabla p|^m\lesssim 1+\norm{\frac{\eta(p)^2}{p\eta'(p)^{\frac{2m-4}{4}}}}_{L^{\frac{m}{4-m}}(Q_T)}\big(\int_{Q_T} p|\Delta p|^2+p|D^2p|^2\big).
\]
as desired.

\end{proof}

The rest of this section will be building towards the  weighted $L^1$ bounds on $\gamma (\Delta p+G)$ (c.f. Proposition \ref{prop:u_l1}).  Most of the effort will be in establishing weighted $L^2$ AB type estimates on $\gamma(\Delta p+G)_-$ which is equivalent to estimating $u_+^2$.  Our estimate of $u_+^2$
 is a modification of the estimate from \cite{gpsg} and related to the weighted estimates in \cite{weighted_ab}.  Instead of directly estimating $u_+^2$, we consider the weighted quantity $\omega(p) u_+^2$ where $\omega:\RR\to\RR$ will be a function that vanishes at 0. 
More specifically, we shall require that our weight $\omega:\RR\to\RR$ satisfies the following properties
\begin{enumerate}[(W1)]
\item $\omega$ is nonnegative, increasing, and concave.
\item $\omega(a)\leq \gamma a\omega'(a)$ for all $a\in [0,p_h]$.
\item There exists a constant $C>0$  such that $\omega(a)\leq Ca^{\gi} \max(\frac{1}{\gamma },a\zeta'(a))$ where $\zeta(a)$ is a nonnegative, increasing function. 
\end{enumerate} 
Note that condition (W3) combined with the integrability properties of $\rho=p^{\gamma}$ implies that $\omega(p)\in L^{\infty}([0,T];L^1(\RR^d))$.
We will keep the weights abstract until our $L^1$ estimate, Proposition \ref{prop:u_l1}, where we will finally fix a choice.

Let us note that the main advantage of working with these weaker weighted quantity is that we can have far less restrictive structural assumptions on the growth terms and our estimates will hold in the incompressible limit $\gamma\to\infty$.  In addition, the calculation itself will be a bit simpler than the one in \cite{gpsg} since we do not need to include a localizing function (more precisely, one can think of $\omega(p)$ as a special choice of a localizing function).  Nonetheless, the calculation is still quite complicated and will be separated into a few different steps. Readers who are just interested in the bound itself can skip to the statements of Propositions \ref{prop:ab} and more importantly \ref{prop:u_l1}.  Readers who are interested in the argument itself will be ``rewarded'' with many ``fun'' (tedious) applications of integration by parts and Young's inequality.  

\begin{lemma}\label{lem:dtu_estimate} Let $f:\RR\to\RR$ be a $C^2$ convex increasing function such that $f(a)=0$ for all $a\leq 0$.  If we let $f^*$ denote the convex conjugate of $f$, then
\begin{multline}\label{eq:lem_dtu}
   \frac{d}{dt}\int_{\RR^d}\gi \omega(p) f(u)+\int_{\RR^d} \omega(p) f^*(f'(u))(\gi u+G)+p\omega(p) f''(u)|\nabla u|^2\\
 \leq \int_{\RR^d} Gf(u)(\frac{2}{\gamma}\omega(p)-p\omega'(p)) + \omega(p) f'(u)(2\nabla G\cdot \nabla p-\partial_t G)
\end{multline}
\end{lemma}
\begin{remark}
Rather than directly work with $f(u)= u_+^2$ we instead consider a more generic function $f$, which makes it easier to see when an integration by parts will be useful and helps us see why we will eventually be forced into the choice $f(u)=u_+^2$. 
\end{remark}

\begin{proof}
Differentiating in time and using (\ref{eq:pme_pressure}), we have
\[
\frac{d}{dt}\int_{\RR^d} \gi \omega(p) f(u)=\int_{\RR^d} \gi \omega'(p)[|\nabla p|^2- p u]f(u)+\omega(p) f'(u) \gi \partial_t u.
\]
Using (\ref{eq:pme_pressure}) again, we see that
\[
\gi \partial_t u= \Delta (pu)-\Delta|\nabla p|^2 -\partial_t G. 
\]
Expanding the terms with the Laplacian, we get
\[
\gi \partial_t u=-(\gi u+G)u+2\nabla u\cdot \nabla p+ p\Delta u-2|D^2p|^2+2 \nabla (\gi u+G)\cdot \nabla p-\partial_t G
\]
Hence, after some rearranging, we have shown that
\begin{multline}\label{eq:dtu1}
   \frac{d}{dt}\int_{\RR^d}\gi \omega(p) f(u)+\int_{\RR^d} \omega(p) uf'(u)(\gi u+G)+2\omega(p) f'(u)|D^2p|^2\\
 =\int_{\RR^d} \gi \omega'(p)\big(|\nabla p|^2-p u\big)f(u)+\omega(p) f'(u)\Big( 2\nabla u\cdot \nabla p+ p\Delta u+2\nabla (\gi u+G)\cdot \nabla p\Big)-\omega(p) f'(u)\partial_t G
\end{multline}
Now we want to move spatial derivatives off of $u$.  Moving $f'(u)$ inside the parentheses, we see that the second term on the right hand side of (\ref{eq:dtu1}) is equal to
\[
\int_{\RR^d} \omega(p) \Big( 2(1+\gi)\nabla f(u)\cdot \nabla p+ p\Delta f(u)- pf''(u)|\nabla u|^2+2f'(u)\nabla G\cdot \nabla p \Big)
\]
Integrating by parts, the previous line is equal to
\begin{equation}\label{eq:dtu2}
\int_{\RR^d}  f(u)\Delta (p\omega(p) ) -2(1+\gi) f(u)\nabla \cdot (\omega(p)\nabla p) - p\omega(p) f''(u)|\nabla u|^2+2\omega(p)f'(u)\nabla G\cdot \nabla p
\end{equation}
Now we expand $\Delta (p\omega(p))=\nabla \cdot (\omega(p)\nabla p)+\nabla \cdot (p\omega'(p)\nabla p)$ to see that
(\ref{eq:dtu2}) is equal to
\begin{equation}\label{eq:dtu3}
\int_{\RR^d} -(1+\frac{2}{\gamma}) f(u)\nabla \cdot (\omega(p)\nabla p)+f(u)\nabla \cdot (p\omega'(p)\nabla p) - p\omega(p) f''(u)|\nabla u|^2+2\omega(p) f'(u)\nabla G\cdot \nabla p
\end{equation}
Plugging this back into (\ref{eq:dtu1}) and rearranging, we have
\begin{multline}\label{eq:dtu4}
   \frac{d}{dt}\int_{\RR^d}\gi \omega(p) f(u)+\int_{\RR^d} \omega(p) uf'(u)(\gi u+G)+2\omega(p) f'(u)|D^2p|^2+\omega(p) f'(u)\partial_t G\\
 =\int_{\RR^d} \gi \omega'(p)\big(|\nabla p|^2-p u\big)f(u)-(1+\frac{2}{\gamma}) f(u)\nabla \cdot (\omega(p)\nabla p)\\ +\int_{\RR^d} f(u)\nabla \cdot (p\omega'(p)\nabla p) - p\omega(p) f''(u)|\nabla u|^2+2\omega(p) f'(u)\nabla G\cdot \nabla p
\end{multline}
Expanding the terms with the divergence operator, we get 
\begin{multline}\label{eq:dtu5}
   \frac{d}{dt}\int_{\RR^d}\gi \omega(p) f(u)+\int_{\RR^d} \omega(p) uf'(u)(\gi u+G)+2\omega(p) f'(u)|D^2p|^2+\omega(p) f'(u)\partial_t G\\
 =\int_{\RR^d} \gi \omega'(p)\big(|\nabla p|^2-p u\big)f(u)+(1+\frac{2}{\gamma}) f(u)\omega(p)(\gi u+G)-\frac{2}{\gamma}f(u)\omega'(p)|\nabla p|^2\\ +\int_{\RR^d} f(u)(p\omega''(p)|\nabla p|^2 -p\omega'(p)(\gi u+G)) - p\omega(p) f''(u)|\nabla u|^2+2\omega(p) f'(u)\nabla G\cdot \nabla p
\end{multline}
Combining similar terms and rearranging, we get 
\begin{multline}\label{eq:dtu6}
   \frac{d}{dt}\int_{\RR^d}\gi \omega(p) f(u)+\int_{\RR^d} \omega(p) (uf'(u)-f(u))(\gi u+G)+2\omega(p) f'(u)|D^2p|^2+p\omega(p) f''(u)|\nabla u|^2\\
 =\int_{\RR^d} \frac{2}{\gamma}\big(\frac{1}{\gamma} \omega(p)-p\omega'(p)\big)uf(u)+Gf(u)(\frac{2}{\gamma}\omega(p)-p\omega'(p)) +f(u)|\nabla p|^2\big(p\omega''(p)-\frac{1}{\gamma}\omega'(p)\big)\\ +\int_{\RR^d}  \omega(p) f'(u)(2\nabla G\cdot \nabla p-\partial_t G)
\end{multline}
Thanks to our assumptions on $f$ and $\omega$ the terms $2\omega(p) f'(u)|D^2p|^2$, $\frac{2}{\gamma}\big(\frac{1}{\gamma} \omega(p)-p\omega'(p)\big)uf(u)$ and $f(u)|\nabla p|^2\big(p\omega''(p)-\frac{1}{\gamma}\omega'(p)\big)$ are all terms with favorable signs.  Dropping these terms, we get 
\begin{multline}\label{eq:lem_dtu}
   \frac{d}{dt}\int_{\RR^d}\gi \omega(p) f(u)+\int_{\RR^d} \omega(p) \big(uf'(u)-f(u)\big)(\gi u+G)+p\omega(p) f''(u)|\nabla u|^2\\
 \leq \int_{\RR^d} Gf(u)(\frac{2}{\gamma}\omega(p)-p\omega'(p)) + \omega(p) f'(u)(2\nabla G\cdot \nabla p-\partial_t G)
\end{multline}
The result now follows from the identity $uf'(u)-f(u)=f^*(f'(u))$. 

\end{proof}

In the next Lemma, we tackle the estimate of the term $(\partial_t G-2\nabla G\cdot p)$.    This term is quite annoying since it has the form of a transport equation along $-2\nabla p$ instead of $-\nabla p$.  Ultimately, we would like to estimate this term in such a way that there are no derivatives on the ratio variables $c_i$. 

\begin{lemma}\label{lem:dtg_estimate}
There exists a constant $C$ depending only on $T$ and the initial data such that for any $\theta>0$
\begin{multline}\label{eq:dtg}
\int_{Q_T} \omega(p) f'(u)(2\nabla G\cdot \nabla p-\partial_t G)\leq C(1+\theta^{-1})  +\int_{Q_T} \theta \omega(p)f'(u)^2\\
+\int_{Q_T} \frac{1}{2} p\omega(p)f''(u)^2|\nabla u|^2+\omega(p)f'(u)(G+B)(\gi u+G)+ uf'(u)p\omega(p)\partial_p G.
\end{multline}
\end{lemma}

\begin{proof}
We recall that $G=\sum_{i=1}^{\ell} c_i G_i(p,n)$.  Hence, 
\[
2\nabla G\cdot \nabla p=\sum_{i=1}^{\ell} 2\nabla c_i\cdot \nabla p G_i(p,n)+2c_i\partial_p G_i(p,n)|\nabla p|^2+2c_i\partial_n G_i(p,n)\nabla p\cdot \nabla n
\]
and
\[\partial_t G=\sum_{i=1}^{\ell} \partial_t c_i G_i(p,n)+c_i\partial_p G_i(p,n)\partial_t p+c_i\partial_n G_i(p,n)\partial_t n.
\]
Using equation (\ref{eq:c_i}), we have
\begin{multline}
2\nabla G\cdot \nabla p-\partial_t G=\sum_{i=1}^{\ell} \nabla c_i\cdot \nabla p G_i(p,n)+2c_i\partial_p G_i(p,n)|\nabla p|^2+2c_i\partial_n G_i(p,n)\nabla p\cdot \nabla n\\+c_i \big(G-G_i(p,n))
-c_i\partial_p G_i(p,n)\partial_t p-c_i\partial_n G_i(p,n)\partial_t n
\end{multline}
Using equation (\ref{eq:pme_pressure}) and grouping similar terms, we get 
\begin{multline}
2\nabla G\cdot \nabla p-\partial_t G=\sum_{i=1}^{\ell} \nabla c_i\cdot \nabla p G_i(p,n)+c_i\partial_p G_i(p,n)\big(|\nabla p|^2+ pu\big)\\+c_i\partial_n G_i(p,n)(2\nabla p\cdot \nabla n-\partial_t n)+c_i( G-G_i(p,n)) G_i(p,n)
\end{multline}

Now we are ready to begin estimating.
Note that
\[
\sum_{i=1}^{\ell} c_i(G-G_i)G_i=(\sum_{i=1}^{\ell} c_i G_i)^2-\sum_{i=1}^{\ell} c_iG_i^2.
\]
Since $\sum_{i=1}^{\ell} c_i=1$, we can use Jensen's inequality to conclude that $\sum_{i=1}^{\ell} c_i( G-G_i(p,n)) G_i(p,n)\leq 0$.  
Hence, returning to our integral, we have
\begin{multline}\label{eq:g_estimate_3}
\int_{Q_T} \omega(p) f'(u)(2\nabla G\cdot \nabla p-\partial_t G)\leq\\ \sum_{i=1}^{\ell}\int_{Q_T} \omega(p) f'(u)\Big(c_i\partial_n G_i(p,n) (2\nabla p\cdot \nabla n-\partial_t n) +c_i\partial_p G_i(p,n)(|\nabla p|^2+ pu)+G_i(p,n)\nabla c_i\cdot \nabla p\Big).
\end{multline}
Now we want to integrate by parts in the final term to eliminate the bad quantity $\nabla c_i$. After doing so, the second line of (\ref{eq:g_estimate_3}) becomes 
\[
 \sum_{i=1}^{\ell}\int_{Q_T} c_i\omega(p) f'(u)\Big(\partial_n G_i(p,n) (\nabla p\cdot \nabla n-\partial_t n) +\partial_p G_i(p,n) pu-G_i(p,n)\Delta p \Big)-c_iG_i(p,n)\nabla p\cdot \nabla (\omega(p)f'(u))
\]
Using $\partial_n G$ and $\partial_p G$ as shorthands for $\sum_{i=1}^{\ell} c_i\partial_n G_i(p,n)$ and $\sum_{i=1}^{\ell} c_i\partial_p G_i(p,n)$ respectively, we can rewrite the previous line as 
\[
\int_{Q_T} \omega(p) f'(u)\Big( (\nabla p\cdot \nabla n-\partial_t n)\partial_n G + pu\partial_p G -G\Delta p \Big)-G\nabla p\cdot \nabla (\omega(p)f'(u))
\]
After replacing $-\Delta p$ by $\gi u+G$ and expanding $\nabla (\omega(p)f'(u))$, our combined work gives us
\begin{multline}\label{eq:g_estimate_3}
\int_{Q_T} \omega(p) f'(u)(2\nabla G\cdot \nabla p-\partial_t G)\leq\\ \int_{Q_T} \omega(p) f'(u)\Big( (\nabla p\cdot \nabla n-\partial_t n)\partial_n G + pu\partial_p G +(\gi u+G)G \Big)-G\omega'(p)f'(u)|\nabla p|^2-G\omega(p)f''(u)\nabla p\cdot \nabla u.
\end{multline}
Although $-G\omega'(p)f'(u)|\nabla p|^2$ appears to be a good term, we also want to handle the case where $G$ can be negative. Let $B=\max_{i\in \{1,\ldots, \ell\}} \sup_{(p,n)\in [0,\infty)^2} |G_i(p,n)|$. Using $B$ and then integrating by parts, we get 
\[
\int_{Q_T}-G\omega'(p)f'(u)|\nabla p|^2\leq B\int_{Q_T}\omega'(p)f'(u)|\nabla p|^2=-B\int_{Q_T}\omega(p)\nabla \cdot (f'(u)\nabla p)
\]
\[
=\int_{Q_T}  B\omega(p)f'(u)(\gi u +G)-B\omega(p)f''(u)\nabla u\cdot \nabla p.
\]
Plugging this estimate into (\ref{eq:g_estimate_3}), we get 
\begin{multline}\label{eq:g_estimate_4}
\int_{Q_T} \omega(p) f'(u)(2\nabla G\cdot \nabla p-\partial_t G)\leq\\ \int_{Q_T} \omega(p) f'(u)\Big( (\nabla p\cdot \nabla n-\partial_t n)\partial_n G + pu\partial_p G +(\gi u+G)(G+B) \Big)-(G+B)\omega(p)f''(u)\nabla p\cdot \nabla u.
\end{multline}
Now we can use Young's inequality to obtain
\begin{multline}\label{eq:g_estimate_5}
\int_{Q_T} \omega(p) f'(u)(2\nabla G\cdot \nabla p-\partial_t G)\leq \int_{Q_T} \theta^{-1}\omega(p)|\partial_n G|(|\nabla p|^4+|\nabla n|^4+|\partial_t n|^2)+|\nabla p|^2\frac{(G+B)^2\omega(p)}{2p}\\ 
+\int_{Q_T} \frac{1}{2} p\omega(p)f''(u)^2|\nabla u|^2 +\theta \omega(p)f'(u)^2+\omega(p)f'(u)(G+B)(\gi u+G)+ uf'(u)p\omega(p)\partial_p G .
\end{multline}
Combining assumption (W3) with Lemma \ref{lem:p_grad_bound}, it follows that $|\nabla p|^2\frac{(G+B)^2\omega(p)}{2p}$ is bounded and only depends on the initial data and $T$.
Our estimates in Lemmas \ref{lem:p_grad_bound}-\ref{lem:nutrient} and Proposition \ref{prop:phessian}  imply that all of the other terms in the first line are bounded and only depend on the initial data and $T$.  Hence, the result follows. 
\end{proof}

At last we obtain the following AB type estimate. 
\begin{prop} \label{prop:ab}There exists a constant $C_{\gamma}(T)$ depending only on $T, \gamma$, and the initial data, such that 
\begin{equation}\label{eq:u_+_bound}
\int_{Q_T} \omega(p) u_+^2\leq C_{\gamma}(T).
\end{equation}
If in addition $G$ satisfies assumption (G4), then $C_{\gamma}(T)=C(T)$ can be taken independently of $\gamma$. 
\end{prop}
\begin{proof}
Combining Lemmas \ref{lem:dtu_estimate} and \ref{lem:dtg_estimate},  there exists a constant $C>0$ depending only on the initial data and $T$ such that for any $\theta>0$
\begin{multline}\label{eq:u_estimate_1}
   \frac{d}{dt}\int_{\RR^d}\gi \omega(p) f(u)+\int_{\RR^d} \omega(p) f^*(f'(u))(\gi u+G)+p\omega(p) f''(u)(1-\frac{1}{2}f''(u))|\nabla u|^2\\
 \leq  C(1+\theta^{-1}) +\int_{\RR^d} Gf(u)(\frac{2}{\gamma}\omega(p)-p\omega'(p)) +   \theta \omega(p)f'(u)^2
+ \omega(p)f'(u)G(\gi u+G)+ uf'(u)p\omega(p)\partial_p G.
\end{multline}
Since we need $1-\frac{1}{2}f''(u)>0$, the fastest growing choice for $f$ is to take $f(u)=u_+^2$. Plugging in this choice, we get 
\begin{multline}\label{eq:u_estimate_2}
   \frac{d}{dt}\int_{\RR^d}\gi \omega(p) u_+^2+\int_{\RR^d} \omega(p) u_+^2(\gi u+G)-2u_+^2p\omega(p)\partial_p G\\
 \leq  C(1+\theta^{-1}) +\int_{\RR^d} Gu_+^2(\frac{4}{\gamma}\omega(p)-p\omega'(p)) +   4\theta \omega(p)u_+^2
+ 2G^2\omega(p)u_+.
\end{multline}
 We use Young's inequality to get $2G^2\omega(p)u_+\leq \theta^{-1}G^4\omega(p)+\theta\omega(p)u_+^2$ and $Gu_+^2\frac{4}{\gamma}\leq \frac{2}{3\gamma}u_+^3+\frac{64}{3}|G|^3.$  Relying on the fact that $\omega(p)\in L^1(Q_T)$ and $G$ is bounded, we can conclude that
\begin{multline}\label{eq:u_estimate_3}
   \frac{d}{dt}\int_{\RR^d}\gi \omega(p) u_+^2+\int_{\RR^d} \omega(p) u_+^2\big(\frac{1}{3\gamma} u+G\big)-2u_+^2p\omega(p)\partial_p G+Gu_+^2p\omega'(p)\\
 \leq  C'(1+\theta^{-1}) +\int_{\RR^d}     5\theta \omega(p)u_+^2
\end{multline}
for some constant $C'$.  Since concavity implies that $p\omega'(p)\leq \omega(p)$, the first result now follows from Gronwall's inequality and our assumptions on the initial data.  

For the second result, if assumption (G4) holds, then we can see from (\ref{eq:u_estimate_3}) that there exists some $\epsilon>0$ independent of $\gamma$ such that
\begin{equation}\label{eq:u_estimate_4}
   \frac{d}{dt}\int_{\RR^d}\gi \omega(p) u_+^2+\int_{\RR^d} \epsilon \omega(p) u_+^2
 \leq  C'(1+\theta^{-1}) +\int_{\RR^d}     5\theta \omega(p)u_+^2.
\end{equation}
By choosing $\theta\leq \epsilon/10$, we obtain 
\begin{equation}\label{eq:u_estimate_5}
   \frac{d}{dt}\int_{\RR^d}\gi \omega(p) u_+^2+\int_{\RR^d} \epsilon \omega(p) u_+^2
 \leq  C'(1+\epsilon^{-1})
\end{equation} 
and the second result now follows.

\end{proof}

We have at last reached the final estimate of this section where we provide a weighted $L^1$ bound on $|u|$.  Crucially, this bound controls both the positive and negative part of $u$, which will allow us to construct both the forward and backward Lagrangian flows along $-\nabla p$ in the next section.

\begin{prop}\label{prop:u_l1}
There exists a constant $C_{\gamma}(T)$ depending only on $T, \gamma$ and the initial data such that
\begin{equation}\label{eq:u_l1}
\int_{Q_T} \rho\log(1+\frac{1}{p}) |u|+\int_{\RR^d\times\{T\}} \rho\log(1+\frac{1}{p})^2 \leq C_{\gamma}(T).
\end{equation}
Furthermore, if $G$ satisfies condition (G4) and the initial data satisfies (ID5), then for any $\lambda'\in (0,1)\cap (0,\lambda]$ there exists a constant $C(T)$ that is independent of $\gamma\in [1,\infty)$ such that
\begin{equation}\label{eq:u_l1_log}
\int_{Q_T} \rho\log(1+\frac{1}{p})^{\lambda'-1} |u|+\int_{\RR^d\times\{T\}} \rho\log(1+\frac{1}{p})^{\lambda'} \leq C(T)
\end{equation}
where $\lambda$ is the constant in (ID5).
\end{prop}
\begin{proof}
We begin by considering the quantity $\int_{Q_T} \rho \eta(p)\log(1+\frac{1}{p})^r|u|$  where $\eta:[0,\infty)\to[0,\infty)$ is a $C^1$ increasing function and $r\geq 0$ is a parameter both of which we will choose later.  Expanding $|u|=2u_+-u$ and using Young's inequality, we see that 
\begin{equation}\label{eq:l1_u_estimate_1}
\int_{Q_T} \rho \eta(p)\log(1+\frac{1}{p})^r|u|\leq \int_{Q_T} \rho\eta(p)\log(1+\frac{1}{p})^{r+1} +\rho\eta(p)\log(1+\frac{1}{p})^{r-1} u_+^2 -\rho\eta(p)\log(1+\frac{1}{p})^{r}u.
\end{equation}
We now focus on the last term 
$-\int_{Q_T}\rho\eta(p)\log(1+\frac{1}{p})^{r}u$.

Using equation (\ref{eq:pme_pressure}), it follows that
\[
- \int_{Q_T}\rho u\eta(p)\log(1+\frac{1}{p})^{r}= \int_{Q_T} \rho \eta(p)\log(1+\frac{1}{p})^{r}\frac{\partial_t p-|\nabla p|^2}{p}= \int_{Q_T}-\rho (\frac{d}{dt}h(p)-\nabla h(p)\cdot \nabla p)
\]
where $h:[0,\infty)\to\RR$ is the antiderivative $h(a)=-\int_1^a \frac{\eta(b)\log(1+\frac{1}{b})^{r}}{b}\, db. $
Integrating by parts, we get 
\[
\int_{\RR^d\times\{T\}}\rho h(p)- \int_{Q_T}\rho u\eta(p)\log(1+\frac{1}{p})^r=\int_{\RR^d} \rho^0h(p^0)- \int_{Q_T} h(p)\big(\nabla \cdot (\rho\nabla p)-\partial_t \rho\big)
\]
\[
=\int_{\RR^d} \rho^0 h(p^0)+ \int_{Q_T} \rho h(p) G
\]
Combining this with (\ref{eq:l1_u_estimate_1}), we have 
\begin{multline}\label{eq:l1_u_estimate_2}
\int_{\RR^d\times\{T\}}\rho h(p)+\int_{Q_T} \rho \eta(p)\log(1+\frac{1}{p})^{r}|u|\leq \\ \int_{\RR^d} \rho^0 h(p^0)+ \int_{Q_T} \rho\big( h(p)G+\eta(p)\log(1+\frac{1}{p})^{r+1}\big) +\rho\eta(p)\log(1+\frac{1}{p})^{r-1} u_+^2
\end{multline}

Now we are ready to make choices for $\eta$ and $r$.  For $\gamma\in [1,\infty)$ the weight $\omega(p)=\gi p^{\gi}=\gi \rho$ satisfies the conditions (W1-W3), thus we can choose $\eta(p)=\gi$.  If we also choose $r=1$, then a direct computation shows that
\[
h(a)=\frac{1}{2\gamma}\log(1+\frac{1}{a})^{2}+\frac{1}{2\gamma}\tilde{h}(a)
\]
where $\tilde{h}$ is a bounded function on $[0,p_h]$.  Plugging our choices into (\ref{eq:l1_u_estimate_2}), we see that
\begin{multline}\label{eq:l1_u_estimate_3}
\int_{\RR^d\times\{T\}}\gi \rho\log(1+\frac{1}{p})^2+\int_{Q_T}\gi \rho \log(1+\frac{1}{p})|u|\lesssim \\ \int_{\RR^d} \gi \rho^0 \log(1+\frac{1}{p})^2+ \int_{Q_T} \gi \rho\log(1+\frac{1}{p})^{2} +\gi \rho u_+^2
\end{multline}
Hence, Gronwall's inequality and Proposition \ref{prop:ab} imply the existence of a constant $C_{\gamma}(T)$ such that 
\begin{equation}
\int_{Q_T} \rho\log(1+\frac{1}{p}) |u|+\int_{\RR^d\times\{T\}} \rho\log(1+\frac{1}{p})^2 \leq C_{\gamma}(T).
\end{equation}

To get a bound that also is valid in the limit $\gamma\to\infty$, let us choose some $\lambda'\in (0,\lambda]\cap (0,1)$ where $\lambda>0$ is the constant in assumption (ID5) and set $\eta(a)=\log(1+\frac{1}{a})^{\lambda'-1-r}$.   With this choice, $ h(a)=-\int_1^a \frac{\log(1+\frac{1}{b})^{\lambda'-1}}{b}\, db $ and once again a direct computation shows that 
\[
 h(a)=\frac{1}{\lambda'}\log(1+\frac{1}{a})^{\lambda'}+\tilde{h}(a)
\]
where $\tilde{h}$ is a function that is bounded on $[0,p_h]$.
Plugging in this choice to (\ref{eq:l1_u_estimate_2}) we get
\begin{multline}\label{eq:l1_u_estimate_4}
\int_{\RR^d\times\{T\}}\rho\log(1+\frac{1}{p})^{\lambda'}+\int_{Q_T} \rho \log(1+\frac{1}{p})^{\lambda'-1}|u|\lesssim \\ \int_{\RR^d} \rho^0 \log(1+\frac{1}{p})^{\lambda'}+ \int_{Q_T} \rho\log(1+\frac{1}{p})^{\lambda'} + \rho\log(1+\frac{1}{p})^{\lambda'-2} u_+^2
\end{multline}

In order to use Proposition \ref{prop:ab} to bound $\rho\log(1+\frac{1}{p})^{\lambda'-2} u_+^2$,
 we need to check if there exists a weight $\omega(p)$ satisfying (W1-W3) such that $C\omega(p)\geq \rho\log(1+\frac{1}{p})^{\lambda'-2}$ for some constant $C$ that is independent of $\gamma$. We shall choose $\omega(p)=p^{\gi}z(p)=\rho z(p)$ where $z$ is a nonnegative increasing function. To ensure that $\omega$ is concave we need 
\begin{equation}\label{eq:omega_concave}
\omega''(a)=\gi a^{\gi-2}\big((\gi-1)z(a)+2 az'(a)+\gamma a^2z''(a))\leq 0
\end{equation}
for all $a\in (0,p_h]$.  We now consider the choice $z(a)=\log(\xi(a)^{-1})^{\lambda'-2}$ where $\xi:[0,\infty)\to[0,\infty)$ is an increasing concave function that is bounded above by $e^{-6}$ such that $\xi(a)=a$ on $[0,\frac{e^{-6}}{2}]$.   Testing this choice, we get 
\begin{multline}
(\gi-1)z(a)+2 az'(a)+\gamma a^2z''(a)=
(\gi-1)\log(\xi(a)^{-1})^{\lambda'-2}+2(2-\lambda')\frac{a\xi'(a)}{\xi(a)}\log(\xi(a)^{-1})^{\lambda'-3}\\-2\gamma(2-\lambda')(\frac{ a\xi'(a)}{\xi(a)})^2\log(\xi(a)^{-1})^{\lambda'-3}\big(1-(3-\lambda')\log(\xi(a)^{-1})^{-1}\big)+\gamma(2-\lambda')\frac{ a^2\xi''(a)}{\xi(a)}\log(\xi(a)^{-1})^{\lambda'-3}.
\end{multline}
Exploiting the concavity of $\xi$ and the upper bound of $e^{-6}$, it follows that 
\begin{multline}
(\gi-1)z(a)+2 az'(a)+\gamma a^2z''(a)\leq 
(\gi-1)\log(\xi(a)^{-1})^{\lambda'-2}+\frac{1}{3}(2-\lambda')\log(\xi(a)^{-1})^{\lambda'-2}\\-\gamma(2-\lambda')(\frac{ a\xi'(a)}{\xi(a)})^2\log(\xi(a)^{-1})^{\lambda'-3},
\end{multline}
which is nonpositive for all $\gamma\geq 3$. 

It is now easy to check that the remaining properties (W1-W2) are satisfied by our choice.  For property (W3), we note that $\omega(a)=a^{\gi}z(a)=a^{\gi}\log(\xi(a)^{-1})^{\lambda'-2}$.  We then have
\[
\log(\xi(a)^{-1})^{\lambda'-2}=\frac{\xi(a)}{(1-\lambda')\xi'(a)}\frac{d}{da}[\log(\xi(a)^{-1})^{\lambda'-1}].
\]
Since $\log(\xi(a)^{-1})^{\lambda'-1}$ is a nonnegative increasing function for $\lambda'<1$ and $\frac{\xi(a)}{\xi'(a)}=a$ on $[0,e^{-6}/2]$ it follows that condition (W3) is satisfied.

Finally, since $\xi$ is increasing and $\xi(a)=a$ on $[0,\frac{e^{-6}}{2}]$ it also follows that there exists a constant $C>0$ such that $\rho\log(1+\frac{1}{p})^{\lambda'-2}\leq C \omega(p).$  Thus, (\ref{eq:u_l1_log}) now follows from (\ref{eq:l1_u_estimate_3}), Gronwall's inequality, and Proposition \ref{prop:ab} (note that (\ref{eq:u_l1_log}) also holds for $\gamma\in [1,3]$ since (\ref{eq:u_l1}) is a strictly stronger bound and $C_{\gamma}(T)$ only blows up as $\gamma\to\infty$).

\end{proof}

 \section{Stability of Lagrangian flows}\label{sec:flows}

Once again, in this Section, we will assume that we are working with smooth solutions $(\rho_{1},\ldots, \rho_{\ell}, p, n)$ to the system (\ref{eq:multi_pme}-\ref{eq:nutrients}).  Thanks to the smoothness of $p$, the regular Lagrangian flow along $-\nabla p$ must exist by classic Cauchy-Lipschitz theory.  Thus, $(\rho_{1},\ldots, \rho_{\ell}, p, n)$ is already a complete Lagrangian solution in the sense of Definition \ref{def:complete_lagrangian_solution}.  Hence, we can freely assume the existence of the forward and backward flow maps $X, Y$ satisfying equations (\ref{eq:x_flow}) and (\ref{eq:y_flow}) respectively.  The main purpose of this Section is to use our bounds from Section \ref{sec:estimates} to show that $X$ and $Y$ satisfy certain quantitative stability bounds (c.f. Proposition \ref{prop:map_stability}).  

 Our stability bounds will compare $X$ and $Y$ to the forward and backward flows $S, Z$ along some vector field $V\in L^{\infty}_{\loc}([0,\infty);L^2(\RR^d))$   with an associated nonnegative density $\mu\in C_{\loc}([0,\infty);L^1(\RR^d)\cap L^{\infty}(\RR^d))$.  Specifically, we shall assume that $S$ and $Z$ satisfy the flow equations
 \begin{equation}\label{eq:s_map}
    S(t,s,x)=x+\int_{s}^{s+t} V(\tau, S(\tau,s,x))\, d\tau \quad \textup{for a.e} \; x\in \spt(\mu(s,x)),
\end{equation}
\begin{equation}\label{eq:z_map}
    Z(t,s,x)=x-\int_{s-t}^s V(\tau, Z(\tau,s,x))\, d\tau \quad \textup{for a.e} \; x\in \spt(\mu(s,x)),
\end{equation}
and there exists a constant $C>0$ such that
\begin{equation}\label{eq:mu_pushforward_bound_s}
  e^{-tC}\mu(s+t,\cdot)  \leq S(t,s,\cdot)_{\#} \mu(s,\cdot)\leq e^{tC}\mu(s+t,\cdot)
\end{equation}
and
\begin{equation}\label{eq:mu_pushforward_bound_z}
  e^{-tC}\mu(s-t,\cdot)  \leq Z(t,s,\cdot)_{\#} \mu(s,\cdot)\leq e^{tC}\mu(s-t,\cdot).
\end{equation}
We will then show that the difference between $X$ and $S$ ( respectively $Y$ and $Z$) on $\min(\mu,\rho)$ can be controlled in terms of the difference between $V$ and $-\nabla p$.   

\begin{remark}
Since we only assume that $V\in L^{\infty}_{\loc}([0,\infty);L^2(\RR^d))$ there is no guarantee that one can find $S, Z$, and $\mu$ satisfying the properties (\ref{eq:s_map}-\ref{eq:mu_pushforward_bound_z}).  Nonetheless, when we use the results of this Section, it will always end up being in cases where we already know that the analogues of $S, Z, \mu$ exist and satisfy the desired properties.
\end{remark}

Let us emphasize that the estimates in this section are heavily inspired by the quantitative estimates on Lagrangian flows from \cite{crippa_delellis}.  The insight in \cite{crippa_delellis} was that certain logarithmic quantities related to the flow maps could be controlled with just Sobolev regularity on the flow field.  Here we introduce doubly logarithmic quantities that can be controlled without needing to bound $D^2p$ in any $L^r$ space.  Specifically, our quantities take the form
\begin{equation}\label{eq:I}
  I_T(t,s):=\int_{\RR^d} \bar{\rho}(s,x)\log\Big(1+p(s+t,X(t,s,x))\log\big(1+\frac{\min(|X(t,s,x)-S(t,s,x)|,1)}{\delta_{2T}}\big)\Big)\, dx,
\end{equation}
and
\begin{equation}\label{eq:J}
  J_T(t,s):= \int_{\RR^d} \bar{\rho}(s,x)\log\Big(1+p(s-t,Y(t,s,x))\log\big(1+\frac{\min(|Y(t,s,x)-Z(t,s,x)|,1)}{\delta_T}\big)\Big)\, dx, 
\end{equation}
where
\begin{equation}
    \bar{\rho}:=\min(\rho,\mu), \quad \delta_T=\Big(\int_{Q_T} \mu|\nabla p+V|^2\Big)^{1/2}.
\end{equation}
In what follows, we will show that our bounds on $p|D^2p|^2$ and weighted $L^1$bounds on $u$ from Section \ref{sec:estimates} are sufficient to control the above integrals.  In particular, our control on $p|D^2p|^2$ will replace the usual need for Sobolev regularity on $\nabla p$, while our control on $|u|$ will help us make sure that we can keep the factor of $p$ attached to $p|D^2p|^2$ in our calculations.   We will then show that bounds on $I$ and $J$ can be used to bound the differences $\sup_{s\leq T}\sup_{t\leq (T-s)}\int \bar{\rho}(s,x)|X(t,s,x)-S(t,s,x)|\, dx$ and $\sup_{s\leq T}\sup_{t\leq s}\int \bar{\rho}(s,x)|Y(t,s,x)-Z(t,s,x)|\, dx$ in terms of $\delta_T$.

Before we get into the main results of this section, we review some important properties of maximal functions. 

\subsection{Maximal functions}

The maximal functions
\begin{equation}\label{eq:maximal_f}
f(t,x):=\sup_{r>0} \frac{1}{|B_r|}\int_{B_r(x)} |\nabla p(t,y)|^2+p(t,y)|D^2p(t,y)|\, dy,
\end{equation}
and
\begin{equation}\label{eq:maximal_g}
g(t,x):=\sup_{r>0} \frac{1}{|B_r|}\int_{B_r(x)} |\nabla p(t,y)|\, dy.
\end{equation}
will play an important role in our calculations.
It is a classical fact \cite{stein_book} that for any $r\in (1,\infty]$
\begin{equation}\label{eq:f_bound}
    \norm{f}_{L^r(\{t\}\times\RR^d)}\lesssim_{r,d} \norm{\nabla p}_{L^{2r}(\{t\}\times\RR^d)}^2+\norm{pD^2p}_{L^r(\{t\}\times\RR^d)}
\end{equation}
and
\begin{equation}\label{eq:g_bound}
    \norm{g}_{L^r(\{t\}\times\RR^d)}\lesssim_{r,d} \norm{\nabla p}_{L^{r}(\{t\}\times\RR^d)}.
\end{equation}
$f$ and $g$ will show up in our estimates through the following crucial bound.
\begin{lemma}
Given any two points $x_1, x_2\in \RR^d$ and any time $t\geq 0$, we have
\begin{equation}\label{eq:p_flow_diff}
    p(t,x_1)|\nabla p(t,x_1)-\nabla p(t,x_2)|\leq |x_1-x_2|\Big( f(t,x_1)+f(t,x_2)+g(t,x_1)^2+g(t,x_2)^2\Big)
\end{equation}
\end{lemma}
\begin{proof}
By the triangle inequality 
\[
 p(t,x_1)|\nabla p(t,x_1)-\nabla p(t,x_2)|\leq |p(t,x_1)\nabla p(t,x_1)-p(t,x_2)\nabla p(t,x_2)|+|p(t,x_1)-p(t,x_2)||\nabla p(t,x_2)|.
\]
Noting that $|D(p\nabla p)|\leq |\nabla p|^2+p|D^2p|$,  one can use standard maximal function theory \cite{stein_book} to obtain the bounds
\[
|p(t,x_1)\nabla p(t,x_1)-p(t,x_2)\nabla p(t,x_2)|\leq |x_1-x_2|(f(t,x_1)+f(t,x_2))
\]
and
\[
|p(t,x_1)-p(t,x_2)|\leq |x_1-x_2|(g(t,x_1)+ g(t,x_2)|.
\]
The result now follows from Young's inequality and the fact that $|\nabla p|\leq g$ pointwise everywhere. 

\end{proof}

\subsection{Quantitative stability}
We are now ready to prove the main results of this section.  We begin with some basic estimates on the flow maps and their pushforwards.

\begin{lemma}\label{lem:basic_flow_estimate}
Let $B'=\max_{i\in \{1,\ldots,\ell\}} \sup_{(p,n)\in [0,\infty)^2} |G_i(p,n)|$ and $B=\max(B',C)$. where $C$ is the constant from (\ref{eq:mu_pushforward_bound_s}).
For any time $s\geq 0$ and $t\leq s$ we have
\begin{equation}\label{eq:x_pushforward_bound}
e^{-Bt}\rho(s+t,x)\leq X(t,s,\cdot)_{\#} \rho(s,x)\leq e^{Bt}\rho(s+t,x),
\end{equation}
\begin{equation}\label{eq:y_pushforward_bound}
e^{-Bt}\rho(s-t,x)\leq Y(t,s,\cdot)_{\#} \rho(s,x)\leq e^{Bt}\rho(s-t,x).
\end{equation}
Furthermore, for any $s, t\geq 0$
\begin{equation}\label{eq:x_speed_bound}
\Big(\int_{\RR^d} \bar{\rho}(s,x)|X(t,s,x)-S(t,s,x)|^2\Big)^{1/2}\leq (te^{tB})^{1/2}\big(\delta_{s+t}+\norm{\rho^{1/2}\nabla p}_{L^2([s,s+t]\times\RR^d)}\big),
\end{equation}
and for any $s\geq 0$ and $t\leq s$
\begin{equation}\label{eq:y_speed_bound}
\Big(\int_{\RR^d} \bar{\rho}(s,x)|Y(t,s,x)-Z(t,s,x)|^2\Big)^{1/2}\leq  (te^{tB})^{1/2}\big(\delta_s+\norm{\rho^{1/2}\nabla p}_{L^2([s-t,s]\times\RR^d)}\big).
\end{equation}
\end{lemma}
\begin{proof}

Since $Y(t,s,\cdot)$ is the inverse of $X(t,s-t,\cdot)$ (\ref{eq:y_pushforward_bound})  will follow from (\ref{eq:x_pushforward_bound}). 
Using the representation formula (\ref{eq:rho_representation}), it is clear  that $X(t,s,\cdot)_{\#}\rho(s,x)\leq e^{Bt}\rho(s+t,x)$ and $e^{-Bt}\rho(s+t,x)\leq X(t,s,\cdot)_{\#}\rho(s,x).$  
The last two bounds follow from Jensen's inequality and the pushforward bounds.
\end{proof}

Most of the action in this section occurs in the following Lemma where we provide bounds on $I$ and $J$.  We will see that our weighted $L^1$ bounds on $u$ guarantee that $I$ and $J$ cannot blow up as fast as $\log(1+\log(1+\delta_T^{-1}))$. This will be enough to conclude stability of the flow maps.
\begin{lemma}\label{lem:map_stability}
For any $T>0$ and any $\lambda'\in (0,1]$ define
\begin{equation}
   \mathcal{B}_{\lambda'}(T)= \norm{\rho\log(1+\frac{1}{p})^{\lambda'-1} u}_{L^1(Q_T)}+p_h^2+\norm{\rho+\mu}_{L^2(Q_T)}\Big(1+\norm{\nabla p}_{L^4(Q_T)}^2+\norm{pD^2p}_{L^2(Q_T)}\Big),
\end{equation}
we then have the estimates
\begin{equation}\label{eq:I_bound}
 \sup_{s,t\leq T}I_T(t,s)\leq \mathcal{B}_{\lambda'}(2T)e^{BT}\log(1+\log(1+\delta_{2T}^{-1}))^{1-\lambda'},
\end{equation}
and
\begin{equation}\label{eq:J_bound}
 \sup_{s\leq T}\sup_{t\leq s} J_T(t,s)\leq \mathcal{B}_{\lambda'}(2T)e^{BT}\log(1+\log(1+\delta_T^{-1}))^{1-\lambda'}.
\end{equation}
\end{lemma}
\begin{proof}
We will provide the argument for the bound on $I$, the bound on $J$ has a nearly identical proof.
To bound $I$, we will proceed by estimating its time derivative with respect to $t$.  Since the expressions are complicated, we will break down the calculation into smaller pieces first by defining the inner logarithm term $L_X(t,s,x):=\log(1+\frac{\min(|X(t,s,x)-S(t,s,x)|,1)}{\delta_{2T}})$.

Differentiating $L_X$ with respect to $t$, we see that
\[
\partial_t L_X(t,s,x)\leq \frac{|\nabla p(s+t, X(t,s,x))+V(s+t, S(t,s,x))|}{\delta_{2T}+|X(t,s,x)-S(t,s,x)|}.
\]
After an application of the triangle inequality, we can bound the previous line by
\[
\frac{|\nabla p(s+t, X(t,s,x))-\nabla p(s+t,S(t,s,x))|}{|X(t,s,x)-S(t,s,x)|}+\frac{|\nabla p(s+t, S(t,s,x))+ V(s+t,S(t,s,x))|}{\delta_{2T}}.
\]
After combining these bounds with (\ref{eq:p_flow_diff}), we can conclude that
\begin{multline}\label{eq:stability_1}
p(s+t,X(t,s,x))\partial_t L_X(t,s,x)\leq 
f(s+t,X(t,s,x))+f(s+t,S(t,s,x))\\+g(s+t,X(t,s,x))^2+g(s+t,S(t,s,x))^2+p(t+s,X(t,s,x))|\nabla p(s+t, S(t,s,x))+ V(s+t,S(t,s,x))|\delta_{2T}^{-1}.
\end{multline}

Next, we calculate
\[
\frac{d}{dt} p(s+t,X(t,s,x))=\partial_t p(s+t,X(t,s,x))-|\nabla p(s+t,X(t,s,x))|^2= -u(s+t,X(t,s,x))p(s+t,X(t,s,x)),
\]
and
\[\frac{d}{dt}\log(1+p(s+t,X(t,s,x)L_X(t,s,x))=\frac{p(s+t,X(t,s,x))\partial_t L_X(t,s,x)+L_X(t,s,x)\frac{d}{dt} p(s+t,X(t,s,x))}{1+p(s+t,X(t,s,x))L_X(t,s,x)}\]
\[
=\frac{p(s+t,X(t,s,x))\partial_t L_X(t,s,x)}{1+p(s+t,X(t,s,x))L_X(t,s,x)}-\frac{L_X(t,s,x) p(s+t,X(t,s,x))u(s+t,X(t,s,x))}{1+p(s+t,X(t,s,x))L_X(t,s,x)}
\]
Thus,
\begin{multline}\label{eq:stability_4}
    \frac{d}{dt}\log(1+p(t+s,X(t,s,x)L_X(t,s,x))\leq  \\ f(s+t,X(t,s,x))+f(s+t,S(t,s,x))+g(s+t,X(t,s,x))^2+g(s+t,S(t,s,x))^2\\+p(t+s,X(t,s,x))|\nabla p(s+t, S(t,s,x))+ V(s+t,S(t,s,x))|\delta_{2T}^{-1}-\frac{L_X(t,s,x) p(s+t,X(t,s,x))u(s+t,X(t,s,x))}{1+p(s+t,X(t,s,x))L_X(t,s,x)}
\end{multline}

Now we note that $I(0,s)=0$. Using the above bounds on the $t$ derivative of the integrand of $I$, we can conclude that for any $t\geq 0$
\begin{multline}\label{eq:stability_6}
 I_T(t,s)\leq \int_{Q_t}\bar{\rho}(s,x)\Big(    f(s+t,X(t,s,x))+f(s+t,S(t,s,x))+g(s+t,X(t,s,x))^2+g(s+t,S(t,s,x))^2\\+p(t+s,X(t,s,x))|\nabla p(s+t, S(t,s,x))+ V(s+t,S(t,s,x))|\delta_{2T}^{-1}\\-\frac{L_X(t,s,x) p(s+t,X(t,s,x))u(s+t,X(t,s,x))}{1+p(s+t,X(t,s,x))L_X(t,s,x)}\Big)\, dx\, dt,
\end{multline}
Using the pushforward bounds from Lemma \ref{lem:basic_flow_estimate} and changing variables in time, it follows that
\begin{multline}\label{eq:stability_8}
 I_T(t,s)\leq e^{tB}\int_{[s,s+t]\times\RR^d} p_h\delta_{2T}^{-1}\mu(\tau,x)|\nabla p(\tau, x)+ V(\tau,x)|+\big(\rho(\tau,x)+\mu(\tau,x)\big)\big( f(\tau,x)+g(\tau,x)^2\big)\, dx\, d\tau,
 \\-e^{tB}\int_{[s,s+t]} \rho(\tau,x)\frac{u(\tau,x)p(\tau,x)L_X(t,s,Y(t,s+t,x))}{1+p(\tau,x)L_X(t,s,Y(t,s+t,x))},
\end{multline}
where we have also used the fact that $Y(t,s+t,x)$ is the inverse of $X(t,s,x)$.  Since $L_X\leq \log(1+\delta_{2T}^{-1})$ and $a\mapsto \frac{a}{1+a}$ is an increasing function, it follows that  
\begin{multline}\label{eq:stability_9}
 I_T(t,s)\leq e^{tB}\int_{[s,s+t]\times\RR^d} p_h\delta_{2T}^{-1}\mu(\tau,x)|\nabla p(\tau, x)+ V(\tau,x)|+\big(\rho(\tau,x)+\mu(\tau,x)\big)\big( f(\tau,x)+g(\tau,x)^2\big)\, dx\, d\tau,
 \\+e^{tB}\int_{[s,s+t]} \rho(\tau,x)\frac{|u(\tau,x)|p(\tau,x)\log(1+\delta_{2T}^{-1})}{1+p(\tau,x)\log(1+\delta_{2T}^{-1})},
\end{multline}
Using the bounds (\ref{eq:f_bound}) and (\ref{eq:g_bound}) and the definition of $\delta_{T}$, we see that 
\begin{multline}\label{eq:stability_10}
I_T(t,s)\lesssim  e^{tB} \norm{\rho u\log(1+\frac{1}{p})^{\lambda'-1}}_{L^1([s,s+t]\times\RR^d)}\norm{\log(1+\frac{1}{p})^{1-\lambda'}\frac{p\log(1+\delta_{2T}^{-1})}{1+p\log(1+\delta_{2T}^{-1})}}_{L^{\infty}([s,s+t]\times\RR^d)}
\\
+e^{tB}p_h\norm{\mu}_{L^2([s,s+t]\times\RR^d)}^{1/2}+
e^{tB}\norm{\rho+\mu}_{L^2([s,s+t]\times\RR^d)}\Big(\norm{\nabla p}_{L^4([s,s+t]\times\RR^d)}^2+\norm{pD^2p}_{L^2([s,s+t]\times\RR^d)}\Big).
\end{multline}
For $b>0$ large, the function $\log(1+\frac{1}{a})^{1-\lambda'}\frac{ab}{1+ab}$ is roughly maximized at $a=1/b$, thus,
\[
\norm{\log(1+\frac{1}{p})^{1-\lambda'}\frac{p\log(1+\delta_{2T}^{-1})}{1+p\log(1+\delta_{2T}^{-1})}}_{L^{\infty}([s,s+t]\times\RR^d)}\lesssim \log(1+\log(1+\delta_{2T}^{-1}))^{1-\lambda'}.
\]
The result now follows from (\ref{eq:stability_10}) and the above bound.

\end{proof}

Now we are ready to establish the stability property.
\begin{prop}\label{prop:map_stability}
If the initial data satisfies (ID1-ID3) and the growth terms satisfy (G1-G3), then for any $T\geq 0$ and $\lambda'\in (0,1]$ there exists a constant $\mathcal{C}_{\gamma, \lambda'}(T)$ depending only on the initial data, $\lambda', \gamma, T$ and $d$ such that
\begin{equation}\label{eq:x_stability_bound}
\sup_{s\leq T}\sup_{t\leq T}\int_{\RR^d} \bar{\rho}(s,x)|X(t,s,x)-S(t,s,x)|\leq \mathcal{C}_{\gamma,\lambda'}(2T)\log(1+\log(1+\delta_{2T}^{-1}))^{-\lambda'/2},
\end{equation}
\begin{equation}\label{eq:y_stability_bound}
\sup_{s\leq T}\sup_{t\leq s}\int_{\RR^d} \bar{\rho}(s,x)|Y(t,s,x)-Z(t,s,x)|\leq \mathcal{C}_{\gamma,\lambda'}(T)\log(1+\log(1+\delta_T^{-1}))^{-\lambda'/2}.
\end{equation}
Additionally, if the growth terms satisfy (G4) the initial data satisfies (ID5), and $\lambda'\in (0,1)\cap (0,\lambda]$ where $\lambda$ is the constant in (ID5), then $\mathcal{C}_{\gamma,\lambda'}(T)$ is independent of $\gamma$.
\end{prop}
\begin{proof}
We provide the proof for (\ref{eq:y_stability_bound}), the argument for (\ref{eq:x_stability_bound}) is essentially identical.

Given $r>0$ let $D_{r}(t,s):=\{x\in \RR^d: |Y(t,s,x)-Z(t,s,x)|>r\}$. 
We can then estimate
\[
\int_{\RR^d} \bar{\rho}(s,x)|Y(t,s,x)-Z(t,s,x)|\leq r \norm{\rho}_{L^{\infty}([0,s];L^1(\RR^d))}+\int_{D_{r}(t,s)}\bar{\rho}(s,x)|Y(t,s,x)-Z(t,s,x)|\, dx
\]
\[
\leq r \norm{\rho}_{L^{\infty}([0,s];L^1(\RR^d))}+\Big(\int_{D_{r}(t,s)}\bar{\rho}(s,x)\, dx\Big)^{1/2}\Big(\int_{\RR^d}\bar{\rho}(s,x)|Y(t,s,x)-Z(t,s,x)|^2\, dx\Big)^{1/2}
\]
From Lemma \ref{lem:basic_flow_estimate}, we 
already have a bound for $\Big(\int_{\RR^d}\bar{\rho}(s,x)|Y(t,s,x)-Z(t,s,x)|^2\, dx\Big)^{1/2}$.  Thus we focus on the other integral.

Fix some $\epsilon>0$ and note that $D_{r}$ is contained in the union  $ D_{r,\epsilon}(t,s)\cup  A_{Y,\epsilon}(t,s)$ where
\[
D_{r,\epsilon}(t,s):=\{x\in D_{r}(t,s):  p(s-t,Y(t,s,x))>\epsilon\}
\]
and
\[
A_{Y,\epsilon}(t,s):=\{x\in \RR^d:  p(s-t,Y(t,s,x))<\epsilon\}
\]
  Using these sets, we see that
\[
\int_{D_{r}(t,s)}\bar{\rho}(s,x)\, dx\leq  \log(1+\epsilon^{-1})^{-\lambda'}\int_{A_{Y,\epsilon}(t,s)}\bar{\rho}(s,x)\log(1+\frac{1}{p(s-t,Y(t,s,x))})^{\lambda'}\, dx
+\int_{D_{r,\epsilon}(t,s)}\bar{\rho}(s,x)\, dx.
\]
Pushing forward by $Y$ in the first integral, we get 
\[
\int_{D_{r}(t,s)}\bar{\rho}(s,x)\, dx\leq  e^{tB}\log(1+\epsilon^{-1})^{-\lambda'}\int_{\RR^d}\rho(s-t,x)\log(1+\frac{1}{p(s-t,x)})^{\lambda'}\, dx
+\int_{D_{r,\epsilon}(t,s)}\bar{\rho}(s,x)\, dx.
\]

To estimate the final integral, we write 
\[
\int_{D_{r,\epsilon}(t,s)}\bar{\rho}(s,x)\, dx\leq \log(1+\epsilon\log(1+\frac{\min(r,1)}{\delta_T}))^{-1} \int_{D_{r,\epsilon}(t,s)} \bar{\rho}(s,x)\log\Big(1+\epsilon\log\big(1+\frac{\min(r,1)}{\delta_T}\big)\Big)
\]
\[
\leq \log(1+\epsilon\log(1+\frac{\min(r,1)}{\delta_T}))^{-1} \int_{D_{r,\epsilon}(t,s)} \bar{\rho}(s,x)\log\Big(1+p(s-t,Y(t,s,x))\log\big(1+\frac{\min(|Y(t,s,x)-Z(t,s,x)|,1)}{\delta_T}\big)\Big),
\]
where we have taken advantage of the definition of $D_{r,\epsilon}(t,s)$ to obtain the last inequality.  Recognizing that the final integral is bounded above by $J_T$, it follows that 
\[
\int_{D_{\lambda,\epsilon}(t,s)}\bar{\rho}(s,x)\, dx\leq \log(1+\epsilon\log(1+\frac{\min(r,1)}{\delta_T}))^{-1} J_T(t,s).
\]
Thus, after combining our work, we see that
\[
\int_{D_{r}(t,s)}\bar{\rho}(s,x)\, dx\leq
\]
\[
 e^{tB}\log(1+\epsilon^{-1})^{-\lambda'}\int_{\RR^d}\rho(s-t,x)\log(1+\frac{1}{p(s-t,x)})^{\lambda'}\, dx+\log(1+\epsilon\log(1+\frac{\min(r,1)}{\delta_T}))^{-1}J_T(t,s).
\]
Using Proposition \ref{prop:u_l1} and Lemma \ref{lem:map_stability}, it follows that 
\[
\int_{D_{r}(t,s)}\bar{\rho}(s,x)\, dx\lesssim \log(1+\epsilon^{-1})^{-\lambda'}+\log(1+\epsilon\log(1+\frac{\min(r,1)}{\delta_T}))^{-1}\log(1+\log(1+\delta_T^{-1}))^{1-\lambda'}
\]

Now we make the choices $r=\delta_T^{1/2}$ and $\epsilon=\log(1+\delta_T^{-1/2})^{-1/2}$.  
Up to constants, the previous line becomes
\[
 \log(1+\log(1+\delta_T^{-1}))^{-\lambda'}
\]
Combining our work, the result follows.
\end{proof}

\section{Compactness}\label{sec:compactness}

In this final section, we will at last construct complete Lagrangian solutions to the system (\ref{eq:multi_pme}-\ref{eq:nutrients}) under our various assumptions on the initial data and structure of the growth terms (c.f. Section \ref{ssec:main_results}). To construct these solutions, we will take a sequence of smooth solutions to (\ref{eq:multi_pme}-\ref{eq:nutrients}) and use our results from Sections \ref{sec:estimates} and \ref{sec:flows} to prove that strong limit points exist and satisfy Definition \ref{def:complete_lagrangian_solution}. We will first construct solutions in the case $\gamma\in [1,\infty)$ and then consider the incompressible limit $\gamma\to \infty$.

\subsection{Compactness for $\gamma$ fixed}\label{ssec:compactness_gamma}

We begin with the following Proposition which guarantees the existence of smooth solutions
under certain assumptions on the initial data and growth terms.  Here the crucial property will be that the initial data is not compactly supported.  Let us emphasize that the existence of smooth solutions for PME equations with data bounded away from zero is very well-known in the literature \cite{vazquez_book}.  We take an approach similar to \cite{gpsg}.

\begin{prop}\label{prop:smooth_solutions}
Let $(\rho_1^0,\ldots, \rho_{\ell}^0, n^0)$ be initial data satisfying (ID1-ID3) and let $G_1,\ldots, G_{\ell}$ be growth terms satisfying  (G1-G3). If $(\rho_1^0,\ldots, \rho_{\ell}^0, n^0)$ and $G_1,\ldots, G_{\ell}$ are smooth and there exists some $\delta>0$ such that $\sum_{i=1}^{\ell} \rho_i^0\geq \delta e^{- |x|^2}$
then there exists a smooth complete Lagrangian solution to (\ref{eq:multi_pme}-\ref{eq:nutrients}) with initial data $(\rho_{1}^0,\ldots, \rho_{\ell}^0, n^0)$ and growth terms $G_{i}$.
\end{prop}
\begin{proof}
We can construct solutions through the following iteration scheme. To initialize the scheme we first set $\rho_{i,0}(t,x)=\rho_i^0(x)$ and $n_0(t,x)=n^0(x)$ for all $(t,x)$, then we set $\rho_0=(\sum_{i=1}^{\ell} \rho_{i,0})$,  $p_0=\rho_0^{\gamma}$,  We then iterate by solving the following equations
\begin{equation}\label{eq:iteration_1}
c_{i,m}=\frac{\rho_{i,m}}{\rho_m}, \quad G^m=\sum_{i=1}^{\ell} c_{i,m}G_i(p_m,n_m),
\end{equation}
\begin{equation}\label{eq:iteration_2}
\partial_t p_{m+1}-\gamma (p_{m+1}+\frac{1}{m})\Delta p_{m+1}-|\nabla p_{m+1}|^2=\gamma p_{m+1}G^m,
\end{equation}
\begin{equation}\label{eq:iteration_3}
\partial_t \rho_{i,m+1}-\nabla \cdot(\rho_{i,m+1}\nabla p_{m+1})=\rho_{i,m+1}G_i(p_{m+1},n_m)
\end{equation}
\begin{equation}\label{eq:iteration_4}
\partial_t n_{m+1}-\alpha \Delta n_{m+1}=-n_{m+1}\sum_{i=1}^{\ell}\beta_i\rho_{i,m+1}
\end{equation}
By construction, each step of the scheme produces a smooth solution (this is clear for (\ref{eq:iteration_2}) and (\ref{eq:iteration_4}) and we then note that (\ref{eq:iteration_3}) is a continuity equation with smooth initial data,  smooth vector field, and smooth source). We can also check that $\delta^{\gamma}\exp(-\gamma|x|^2-\theta t)$ is a subsolution to (\ref{eq:iteration_2}) once $\theta$ is chosen to be sufficiently large. Thus, $p_{m}(t,x)\geq  \delta^{\gamma}\exp(-\gamma|x|^2-\theta t)$ for all $m$.  This guarantees that on balls of finite radius equation (\ref{eq:iteration_2}) is uniformly parabolic independently of $m$.  From here, the convergence of the scheme to a smooth solution is well-known in the parabolic literature folklore.
\end{proof}

Given initial data $(\rho_{1}^0, \ldots, \rho_{\ell}^0, n^0)$ satisfying (ID1-ID3) and growth terms $G_i$ satisfying (G1-G3), we want to use the previous Proposition to  construct a sequence of smooth solutions that will converge to a complete Lagrangian solution with the desired initial data. Let $\eta:\RR^d\to\RR$ be a smooth compactly supported mollifier. For each $k\in \ZZ_+$ we define
\begin{equation}
    \rho_{i,k}^0:=\frac{1}{k}e^{-|x|^2}+\eta_{\frac{1}{k}}*\rho_i^0,
\end{equation}
\begin{equation}
    n_k^0=\eta_{\frac{1}{k}}*n^0,
\end{equation}
and we choose $G_{i,k}$ to be a sequence of smooth approximations to $G_i$.  By Proposition \ref{prop:smooth_solutions}, for each $k\in \ZZ_+$, there exists a smooth solution $(\rho_{1,k}, \ldots, \rho_{\ell,k}, p_k, n_k)$ to (\ref{eq:multi_pme}-\ref{eq:nutrients}) with initial data $(\rho_{1,k}^0, \ldots, \rho_{\ell,k}^0, n_k^0)$.   Note that the smoothness of the variables implies that $(\rho_{1,k}, \ldots, \rho_{\ell,k}, p_k, n_k)$ is a complete Lagrangian solution to the system. Hence, we are assured the existence of the forward and backward flow maps $X_k, Y_k$ along $-\nabla p_k$ satisfying all of the properties in Definition \ref{def:complete_lagrangian_solution}.  It remains to verify that these sequences have sufficient compactness to extract limit points and prove that the limit points are the desired complete Lagrangian solutions.

\begin{lemma}\label{lem:precompact}
Fix some $T\geq 0$.
Both $p_k$ and $n_k$ are $L^2([0,T];H^1(\RR^d))\cap C([0,T];L^2(\RR^d))$ strongly precompact.  For each $i\in \{1,\ldots, \ell\}$ the family $\rho_{i,k}$ is $L^1(Q_T)\cap W^{1,2}([0,T];H^{-1}(\RR^d))$ weakly precompact. 
\end{lemma}
\begin{proof}

Thanks to Proposition \ref{prop:u_l1},
we have  \[
\sup_k \int_{Q_T} \rho_k|\partial_t p_k|+\rho_k|\Delta p_k|=\sup_k \int_{Q_T} \rho_k|\nabla p|^2+\rho_k p_k|u_k|+\gi \rho_k| u_k|+\rho_k|G_k|<\infty
\]
We can also compute
\[
\sup_k\int_{Q_T} |x|^{1/2}|\nabla p_k(t,x)|^2\leq \sup_k\int_{Q_T} \rho_k|x|^{2}+\frac{|\nabla p_k(t,x)|^{8/3}}{\rho_k^{1/3}}<\infty.
\]
If we define $q_k=p_k^{1+\gi}$, then $|\partial_t q_k|=\rho_k|\partial_t p_k|$, $(1+|x|^{1/2})|\nabla q_k|^2=(1+|x|^{1/2})\rho_k^2|\nabla p_k|^2$ and $|\Delta q_k|\leq |\nabla \rho_k\cdot \nabla p_k|+\rho_k|\Delta p_k|$. Hence, $q_k$ is $L^2([0,T];H^1(\RR^d))$ precompact.

Now we want to transfer these precompactness properties to $p_k$.  We need to be a little careful since the transformation $a\mapsto a^{(1+\gi)^{-1}}$ is not $C^1$.  Let $k_j$ be a subsequence such that $q_{k_j}$ is  $L^2([0,T];H^1(\RR^d))$ Cauchy. Fix some $\epsilon>0$ and let $\chi_{\epsilon}:\RR\to\RR$ be the characteristic function of $[0,\epsilon]$. We can then compute
\[
\norm{\nabla p_{k_j}-\nabla p_{k_m}}_{L^2(Q_T)}\leq \norm{\nabla p_{k_j}\chi_{\epsilon}(p_{k_j})}_{L^2(Q_T)}+\norm{\nabla p_{k_m}\chi_{\epsilon}(p_{k_m})}_{L^2(Q_T)}
\]
\[
+\norm{\nabla p_{k_j}(1-\chi_{\epsilon}(p_{k_j}))-\nabla p_{k_m}(1-\chi_{\epsilon}(p_{k_m}))}_{L^2(Q_T)}
\]
Hence, 
\[
\lim_{j, m\to\infty} \norm{\nabla p_{k_j}-\nabla p_{k_m}}_{L^2(Q_T)}\leq 2\sup_{j}\norm{\nabla p_{k_j}\chi_{\epsilon}(p_{k_j})}_{L^2(Q_T)}
\]
\[
\leq 2\epsilon^{1/4}\sup_{j}\norm{p_{k_j}^{-1/4}\nabla p_{k_j}}_{L^2(Q_T)}\lesssim \epsilon^{1/4}.
\]
Hence, the $p_k$ are  $L^2([0,T];H^1(\RR^d))$ precompact.  

To get precompactness in $C([0,T];L^2(\RR^d))$ we note that $|\partial_t p_k|^2\leq 2|\nabla p_k|^4+2p^2_ku_k^2$, hence $\partial_t p_k$ is uniformly bounded in $L^2(Q_T)$ thanks to Proposition \ref{prop:phessian}. Now the precompactness in $C([0,T];L^2(\RR^d))$ follows from the precompactness in $L^2([0,T];H^1(\RR^d))$.

The precompactness of the nutrients is clear from the uniform $L^2$ bounds on $\partial_t n_k$ and $\Delta n_k$.
 The weak precompactness of the $\rho_{i,k}$ in $L^1(Q_T)\cap W^{1,2}([0,T];H^{-1}(\RR^d))$ follows from the bound $0\leq \rho_{i,k}\leq \rho_k$ and the equation (\ref{eq:multi_pme}).
\end{proof}

Now that we have established the precompactness of the family of smooth solutions, we can deduce the existence of a limit point $(\rho_1,\ldots, \rho_{\ell}, p,n)$.  In what follows, we shall assume that we have extracted a subsequence (that we do not relabel) such that $(\rho_{1,k}, \ldots, \rho_{\ell,k}, p_k, n_k)$ converges to $(\rho_1, \ldots, \rho_{\ell}, p,n)$ with the various notions of convergence identified in Lemma \ref{lem:precompact}.    We will then show that this limit point is a complete Lagrangian solution and prove Theorem \ref{thm:main1}.

First we show that the maps $X_k, Y_k$ are Cauchy on the support of $\rho$.

\begin{lemma}\label{lem:flow_cauchy}
For any $s, T, M\geq 0$,
\[
\sup_{t\leq T} \lim_{j,k\to\infty}  \int_{\RR^d} \rho(s,x)\min(|X_{j}(t,s,x)-X_{k}(t,s,x)|,M)\, dx=0.
\]
\[
\sup_{t\leq s}\lim_{j,k\to\infty}  \int_{\RR^d} \rho(s,x)\min(|Y_{j}(t,s,x)-Y_{k}(t,s,x)|,M)\, dx=0.
\]
\end{lemma}
\begin{proof}
The strong convergence of $p_k$ to $p$ in $C([0,T];L^1(\RR^d))$ implies that $\rho_k$ converges strongly to $\rho$ in $C([0,T];L^r(\RR^d))$.  
Hence, we can replace $\rho$ in the above integrals with $\min(\rho_k(s,x), \rho_j(s,x))$.  
If we allow $(-\nabla p_{k}, \rho_{k})$ and $(-\nabla p_{j}, \rho_{j})$ to respectively play the roles of $(-\nabla p, \rho)$ and $(V, \mu)$ in  Proposition \ref{prop:map_stability}, the result follows from the vanishing of $\norm{\nabla p_{k}-\nabla p_{j}}_{L^2(Q_T))}$ as $j,k\to\infty$. 
\end{proof}
The strong convergence of the flow maps in Lemma \ref{lem:flow_cauchy} implies the existence of the forward and backward Lagrangian flow maps $X, Y$ along $-\nabla p$. The strong convergence guarantees that these maps satisfy all of the properties in requirement (\ref{def:complete_lagrangian_solution_2}) of the definition of complete Lagrangian solutions (i.e. the flow equations, semigroup property, and inversion formulas). Hence, we have almost succeeded in constructing our desired solution.
Before we prove Theorem \ref{thm:main1}, we establish two uniqueness properties for the flow along $-\nabla p$.  First we show that the flow maps $X, Y$ have a stability property a l\'a  Proposition \ref{prop:map_stability} and then we show  that solutions to the continuity equation along $-\nabla p$ are unique provided that the density stays within the support of $\rho$.

\begin{prop}\label{prop:flow_uniqueness}
 Let $X$ and $Y$ be the $L^1_{\loc}([0,\infty);L^1(\rho))$ limits of $X_{k}$ and $Y_{k}$. Let $V, \mu, S, Z$ be as in Section \ref{sec:flows}. For any $T\geq 0$, we have the estimates
 \begin{equation}\label{eq:x_final_stability_bound}
\sup_{s\leq T}\sup_{t\leq T}\int_{\RR^d} \min(\mu(s,x), \rho(s,x))|X(t,s,x)-S(t,s,x)|\leq \mathcal{C}_{\gamma}(2T)\log(1+\log(1+\delta_T^{-1}))^{-1/2},
\end{equation}
\begin{equation}\label{eq:y_final_stability_bound}
\sup_{s\leq T}\sup_{t\leq s}\int_{\RR^d} \min(\mu(s,x),\rho(s,x))|Y(t,s,x)-Z(t,s,x)|\leq \mathcal{C}_{\gamma}(T)\log(1+\log(1+\delta_T^{-1}))^{-1/2},
\end{equation}
where $\delta_T=\norm{\nabla p+V}_{L^2(Q_T)}$ and $\mathcal{C}_{\gamma}(T)$ is the same constant as in Proposition \ref{prop:map_stability}.
 \end{prop}
 \begin{proof}
Arguing as in Lemma \ref{lem:flow_cauchy}, it follows that for any $M\geq 0$
 \[
 \int_{\RR^d} \min(\mu(s,x), \rho(s,x))\min(|X(t,s,x)-S(t,s,x)|,M)\, dx=
 \]
 \[
 \lim_{k\to\infty} \int_{\RR^d} \min(\mu(s,x), \rho_k(s,x))\min(|X_k(t,s,x)-S(t,s,x)|,M)\, dx\leq 
 \]
 \[
\lim_{k\to\infty}\sup_{s\leq T}\sup_{t\leq T}\int_{\RR^d}  \min(\mu(s,x), \rho_k(s,x))|X_k(t,s,x)-S(t,s,x)|\leq \mathcal{C}_{\gamma}(2T)\log(1+\log(1+\delta_{T}^{-1}))^{-1/2}.
\]
 Sending $M\to\infty$ the result follows.

 \end{proof}

\begin{prop}\label{prop:continuity_uniqueness}
 Let $X$ and $Y$ be the $L^1_{\loc}([0,\infty);L^1(\rho))$ limits of $X_{k}$ and $Y_{k}$.
Suppose that $\nu\in  L^{\infty}_{\loc}([0,\infty);L^1(\RR^d)\cap L^{\infty}(\RR^d))$ is a weak solution to the continuity equation 
\begin{equation}\label{eq:nu_continuity}
\partial_t\nu-\nabla \cdot (\nu\nabla p)=0,
\end{equation}
with initial data $\nu^0$.   If $\nu$ is everywhere nonnegative and $\nu(s,\cdot)$ is absolutely continuous with respect to $\rho(s,\cdot)$
for all $s\geq 0$, then
$X(t,s,\cdot)_{\#}\nu(s,\cdot)=\nu(s+t,\cdot)$ for all $s, t\geq 0$ almost everywhere in space and $Y(t,s,\cdot)_{\#}\nu(s,\cdot)=\nu(s-t,\cdot)$ for every $s\geq 0$ and $t\leq s$ almost everywhere in space.
\end{prop}
\begin{remark}
This proposition gives another way to argue that $X$ and $Y$ are the unique forward and backward flow maps along $-\nabla p$ when restricted to the support of $\rho$ (see Ambrosio's superposition principle \cite{ambrosio_2008}). 
\end{remark}
\begin{proof}
Fix a time step $\tau>0$ and for each $k\in \ZZ_+$, we construct the optimal transport  interpolants between $\nu(\tau k, \cdot)$ and $\nu(\tau (k+1),\cdot)$ associated to the quadratic transportation cost.  Doing so, we obtain $\nu^\tau, \vp^{\tau}, S^{\tau}$ where $\nu^{\tau}$ is a density such that $\nu^{\tau}(k\tau,\cdot)=\nu(k\tau,\cdot)$ for all $k$, $\nu^{\tau}, \vp^{\tau}$ are weak solutions to the continuity equation
\begin{equation}\label{eq:nu_tau_continuity}
\partial_t \nu^{\tau}+\nabla \cdot(\nu^{\tau}\nabla \vp^{\tau})=0,
\end{equation}
and $S^{\tau}$ satisfies
\[
\partial_t S^{\tau}(t,s,x)=\nabla \vp^{\tau}(s+t,S^{\tau}(t,s,x)), \quad S^{\tau}(t,s,\cdot)_{\#}\nu^{\tau}(s,\cdot)=\nu^{\tau}(s+t,\cdot),
\]
see for instance \cite{otam}.
Furthermore, if we define $m^{\tau}=\nu^{\tau}\nabla \vp^{\tau}$, then for any $j\in \ZZ_+$ we have
\[
\int_{0}^{j\tau}\int_{\RR^d} \frac{|m^{\tau}|^2}{2\nu^{\tau}}=\inf_{(\mu, b)\in \mathcal{C}_{\tau}} \int_{0}^{\tau j}\int_{\RR^d} \frac{|b|^2}{2\mu}
\]
where $C_{\tau}$ is the space of all density-flux pairs $(\mu, b)\in L^1_{\loc}(Q_{\infty})\times L^2_{\loc}(Q_{\infty})$ that are weak solutions to the continuity equation $\partial_t \mu+\nabla \cdot b=0$ such that $\mu(\tau k, \cdot)=\nu(\tau k, \cdot)$ for all $k\in \ZZ_+$. Note that $(\nu, \nu\nabla p)\in C_{\tau}$ for any choice of $\tau$, hence,
\[
\int_{0}^{j\tau}\int_{\RR^d} \frac{|m^{\tau}|^2}{2\nu^{\tau}}\leq \int_{0}^{j\tau}\int_{\RR^d} \frac{\nu}{2}|\nabla p|^2.
\]

Given any $\psi\in H^1_c(Q_{\infty})$ and $j\in \ZZ_+$, we have \[
\int_{0}^{j\tau}\int_{\RR^d} (\nu-\nu^{\tau})\psi=\int_{\RR^d}\sum_{k=0}^{j-1}\int_{k\tau}^{(k+1)\tau} \int_{k\tau}^s \Big(\nu^{\tau}(\theta,x)\nabla \vp^{\tau}(\theta,x)-\nu(\theta,x)\nabla p(\theta,x)\Big)\cdot \nabla \psi(\theta,x)\, d\theta \, ds\, dx,
\]
thus it follows that $\nu^{\tau}$ converges to $\nu$ in $\dot{H}^{-1}_{\loc}(Q_{\infty})$ as $\tau\to 0$. Hence, for any $\psi\in H^1_c(Q_{\infty})$ it follows from (\ref{eq:nu_continuity}) and (\ref{eq:nu_tau_continuity}) that
\[
\lim_{\tau\to 0} \int_{Q_{\infty}} (m^{\tau}-\nu\nabla p)\cdot \nabla \psi=0
\]
so $m^{\tau}$ converges weakly to $\nu\nabla p+w$ where $w$ is some divergence free vector field.  Given some $T>0$ let $j_{\tau}=\lceil\frac{T}{\tau}\rceil$.  We can then compute
\[
\int_{0}^{j_{\tau}\tau}\int_{\RR^d} \frac{1}{2}\nu^{\tau}|\nabla \vp^{\tau}-\nabla p|^2=\int_{0}^{j_{\tau}\tau}\int_{\RR^d} \frac{|m^{\tau}|^2}{2\nu^{\tau}}-m^{\tau}\cdot \nabla p+\frac{1}{2}\nu^{\tau}|\nabla p|^2.
\]
\[
\leq \int_{0}^{j_{\tau}\tau}\int_{\RR^d} \frac{\nu}{2}|\nabla p|^2-m^{\tau}\cdot \nabla p+\frac{1}{2}\nu^{\tau}|\nabla p|^2.
\]
 $\nu^{\tau}$ must converge weakly to $\nu$ in $L^2_{\loc}([0,\infty);L^2(\RR^d))$, therefore
\[
\lim_{\tau\to 0}\int_0^{j_{\tau}\tau}\int_{\RR^d} \frac{1}{2}\nu^{\tau}|\nabla \vp^{\tau}-\nabla p|^2=0.
\]
Now we can use Proposition \ref{prop:flow_uniqueness} to deduce that for any $s,t\geq 0$ \begin{equation}\label{eq:continuity_1}\lim_{\tau\to 0} \int_{\RR^d}\min(\nu^{\tau}(s,x),\rho(s,x)) |S^{\tau}(t,s,x)-X(t,s,x)|\, dx=0.
\end{equation}

Finally, we can establish the pushforward formulas for $\nu$.
Let $\vp:\RR^d\to\RR$ be a smooth compactly supported test function.  We can compute
\[
|\int_{\RR^d} \big(\nu^{\tau}(s+t,x)-X(t,s,\cdot)_{\#}\nu(s,\cdot)\big)\vp(x)\, dx|=
\]
\[
|\int_{\RR^d} \nu^{\tau}(s,x)\vp(S^{\tau}(t,s,x))-\nu(s,x)\vp(X(t,s,x))\, dx|=
\]
\[
|\int_{\RR^d} \nu^{\tau}(s,x)\Big(\vp(S^{\tau}(t,s,x))-\vp(X(t,s,x))\Big)\, dx|
\]
Fix some $\epsilon>0$ and let $\Omega_{\epsilon}(s)=\{x\in \RR^d: \rho(s,x)<\epsilon\}$.  The previous line is then bounded from above by 
\begin{multline}\label{eq:continuity_2}
\epsilon^{-1}\norm{\nabla \vp}_{L^{\infty}(\RR^d)}\int_{\RR^d} \nu^{\tau}(s,x)\rho(s,x)|S^{\tau}(t,s,x))-X(t,s,x)|\, dx\\
+2\norm{\vp}_{L^{\infty}(\RR^d)}\int_{\Omega_{\epsilon}(s)}\nu^{\tau}(s,x)\, dx.
\end{multline}
Sending $\tau\to 0$ we see that (\ref{eq:continuity_2}) is equal to
\[
2\norm{\vp}_{L^{\infty}(\RR^d)}\int_{\Omega_{\epsilon}(s)}\nu(s,x)\, dx.
\]
Thanks to our assumption that $\nu$ is absolutely continuous with respect to $\rho$, this last integral vanishes as $\epsilon\to 0$.  Thus it follows that $X(t,s,\cdot)_{\#}\nu(s,\cdot)=\nu(s+t,\cdot)$ almost everywhere in space and for every $s, t\geq 0$.  Since $Y(t,s,x)$ is the inverse of $X(t,s-t,x)$ we also have $Y(t,s,\cdot)_{\#}\nu(s,\cdot)=\nu(s-t,\cdot)$ almost everywhere in space and for every $s\geq 0$ and $t\leq s$.

\end{proof}

Now we can prove Theorem \ref{thm:main1}. 

\begin{proof}[Proof of Theorem \ref{thm:main1}]
From the convergence and uniqueness properties that we have established above, it is clear that $(\rho_{1}, \ldots, \rho_{\ell}, p,n)$ is a complete Lagrangian solution to the tumor growth system.

It remains to prove the nonmixing property.
Let $\rho_{i,j}=\min(\rho_i,\rho_j)$.  From the pushforward representation formula (\ref{eq:rho_representation}), it follows that $Y(t,t,\cdot )_{\#} \rho_{i,j}(t,\cdot)\leq e^{tB}\min(\rho_i^0,\rho_j^0).$  If $\min(\rho_i^0,\rho_j^0)=0$, then it follows that $\rho_{i,j}=0$.  Hence the nonmixing property holds.
\end{proof}

\subsection{The incompressible limit}\label{ssec:compactness_incompressible}

Now we want to construct solutions in the case $\gamma=\infty$.  Given growth terms $G_1,\ldots, G_{\ell}$ satisfying (G1-G4) and initial data $(\rho_1^0,\ldots, \rho_{\ell}^0, n^0)$ satisfying (ID1-ID5), we create a modified sequence of initial data as follows.  Recall that the initial pressure $p^0$ must solve the equation
\begin{equation}
    \Delta p^0+\sum_{i=1}^{\ell} \frac{\rho_i^0}{\rho^0}G_i(p^0,n^0)=0, \quad p^0(1-\rho^0)=0.
\end{equation}
 Using $p^0$, we create the sequence by setting $\rho^0_{\gamma}=(p^0)^{\gi}$ and $\rho_{i,\gamma}^0=\rho^0_{\gamma}\frac{\rho_i^0}{\rho_0}$. 
The key properties of this sequence are summarized below.
\begin{lemma}
$(\rho_{1,\gamma}^0,\ldots, \rho_{\ell,\gamma}, n^0)$ satisfies (ID1-ID3), 
\begin{equation}\label{eq:gamma_poisson_condition}
\sup_{\gamma} \int_{\RR^d} \gamma\rho^0_{\gamma}(\Delta p_0+\sum_{i=1}^{\ell} \frac{\rho_{i,\gamma}^0}{\rho_{\gamma}^0}G_i(p^0,n^0))_{-}^2=0,
\end{equation}
and $\lim_{\gamma\to \infty} \norm{\rho_i^0-\rho_{i,\gamma}^0}_{L^1(\RR^d)}=0.$
\end{lemma}
\begin{proof}
The first two claims are clear from our construction. For the last property, we note that (ID5) implies the existence of some $\lambda>0$ such that $\int_{\RR^d} \rho^0\log(1+1/p^0)^{\lambda}<\infty$, thus, for any $\epsilon>0$ the set $\{x\in \RR^d: p^0<\epsilon\}$ has $\rho^0$ measure at most $\frac{1}{|\log(\epsilon)|^{\lambda}}$.  Thus $\lim_{\gamma\to \infty} (p^0(x))^{\gi}=1$ almost everywhere on the support of $\rho^0$.  Since $p^0=0$ on the complement of the support of $\rho^0$, we can deduce $\lim_{\gamma\to \infty} \norm{\rho_i^0-\rho_{i,\gamma}^0}_{L^1(\RR^d)}=0$  from dominated convergence.
\end{proof}

Now that we have a sequence of initial data satisfying (ID1-ID3), for each $\gamma\geq 1$ we can use Theorem \ref{thm:main1} to construct complete Lagrangian solutions $(\rho_{1,\gamma}, \ldots, \rho_{\ell,\gamma}, p_{\gamma}, n_{\gamma})$ to (\ref{eq:multi_pme}-\ref{eq:nutrients}) with initial data 
$(\rho_1^0,\ldots, \rho_{\ell}^0, n^0)$. Our goal is now to show that these solutions converge to a complete Lagrangian solution to the incompressible system as we send $\gamma\to\infty$.
Due to the fact that we only have uniform $L^1$ regularity for the time derivative of the pressure along the sequence, we will need to proceed more carefully than we did in the case $\gamma<\infty$.

\begin{lemma}\label{lem:incompressible_precompact}
If $\gamma_k$ is a sequence  such that $\lim_{k\to\infty} \gamma_k=\infty$ then $p_{\gamma_k}$ is precompact in $L^2([0,T];H^1(\RR^d))$,  $n_{\gamma_k}$ is precompact in $L^2([0,T];H^1(\RR^d))\cap C([0,T];L^2(\RR^d))$, and $\rho_{i,\gamma_k}$ is weakly precompact in $L^1(Q_T)$ for each $i\in \{1,\ldots, \ell\}$. Furthermore, $\rho_{\gamma_k}$ is precompact in $C([0,T];L^1(\RR^d))$.
\end{lemma}
\begin{proof}
For the first properties, we can argue as in Lemma \ref{lem:precompact}, except that we can no longer establish that $p_{\gamma_k}$ is $C([0,T];L^2(\RR^d))$ precompact.

To establish the precompactness of $\rho_{\gamma_k}$ in $C([0,T];L^1(\RR^d))$, we begin by showing that $\rho_{\gamma_k}$ is spatially equicontinuous in $C([0,T];L^1(\RR^d))$. Note that for any $t\geq 0$, 
\[
\int_{\RR^d\times\{t\}} p_{\gamma_k}^2|\nabla \rho_{\gamma_k}|^2=\int_{\RR^d\times\{t\}} \frac{1}{\gamma_k^2} \rho_{\gamma_k}^2|\nabla p_{\gamma_k}|^2
\]
which is uniformly bounded with respect to $\gamma_k$. Fix some $\epsilon>0$ and let $\eta:\RR\to\RR$ be a smooth increasing function such that $\eta(a)=0$ if $a\leq \frac{1}{2}$ and $\eta(a)=1$ if $a\geq 1$.  Define $\eta_{\epsilon}(a):=\eta(\frac{a}{\epsilon})$.   For any $t\geq 0$ and $y\in \RR^d$
\[
\int_{\RR^d} |\rho_{\gamma_k}(t,x)-\rho_{\gamma_k}(t,x+y)|\, dx\leq
\]
\[
\int_{\RR^d} 2|(1-\eta_{\epsilon}\big(p_{\gamma_k}(t,x)\big)\rho_{\gamma_k}(t,x)|+|\eta_{\epsilon}\big(p_{\gamma_k}(t,x)\big)\rho_{\gamma_k}(t,x)-\eta_{\epsilon}\big(p_{\gamma_k}(t,x+y)\big)\rho_{\gamma_k}(t,x+y)|\, dx.
\]
Taking $\lambda'$ to be the constant in Proposition \ref{prop:u_l1}, we have the bound
\[
\int_{\RR^d} 2|(1-\eta_{\epsilon}\big(p_{\gamma_k}(t,x)\big)\rho_{\gamma_k}(t,x)|\leq 2\log(1+\epsilon^{-1})^{-\lambda'}\int_{\RR^d\times \{t\}} \rho_{\gamma_k}\log(1+\frac{1}{p_{\gamma_k}})^{\lambda'}.
\]
We can also estimate 
\[
\int_{\RR^d} |\eta_{\epsilon}\big(p_{\gamma_k}(t,x)\big)\rho_{\gamma_k}(t,x)-\eta_{\epsilon}\big(p_{\gamma_k}(t,x+y)\big)\rho_{\gamma_k}(t,x+y)|\, dx 
\]
\[
\leq |y|\int_{\RR^d}\int_0^1 \eta'_{\epsilon}\big(p_{\gamma_k}(t,x+ay)\big)\rho_{\gamma_k}(t,x+ay)|\nabla p_{\gamma_k}(t,x+ay)|+\frac{1}{\epsilon^2} p_{\gamma_k}(t,x+ay)^2|\nabla \rho_{\gamma_k}(t,x+ay)|\, da\, dx.
\]
Hence, it follows from our work above that 
\[
\lim_{|y|\to 0} \sup_{k} \int_{\RR^d} |\rho_{\gamma_k}(t,x)-\rho_{\gamma_k}(t,x+y)|\, dx\leq  2\log(1+\epsilon^{-1})^{-\lambda'} \sup_k\int_{\RR^d\times \{t\}} \rho_{\gamma_k}\log(1+\frac{1}{p_{\gamma_k}})^{\lambda'}\leq C\log(1+\epsilon^{-1})^{-\lambda'}
\]
Sending $\epsilon\to 0$, we have established the spatial equicontinuity of $\rho_{\gamma_k}$ in $C([0,T];L^1(\RR^d))$.

Now let $K$ be a spatial mollifier and consider $\rho_{\epsilon,\gamma_k}=K_{\epsilon}*\rho_{\gamma_k}$. The spatial equicontinuity of $\rho_{\gamma_k}$ in $C([0,T];L^1(\RR^d))$ implies that
\begin{equation}\label{eq:spatial_equicontinuity}
\lim_{\epsilon\to 0}\sup_k \norm{\rho_{\gamma_k}-\rho_{\epsilon,\gamma_k}}_{C([0,T];L^1(\RR^d))}=0.
\end{equation}
From equation (\ref{eq:pme}), it is also clear that 
\[
\sup_k \norm{\partial_t \rho_{\gamma_k,\epsilon}}_{C([0,T];L^{2}(\RR^d))}\lesssim_{\epsilon} \sup_k \norm{\nabla p_{\gamma_k}}_{L^{\infty}([0,T];L^2(\RR^d))}.
\] The Aubin-Lions Lemma now implies that $\rho_{\epsilon,\gamma_k}$ is precompact in $C([0,T];L^1(\RR^d))$. 
The precompactness of $\rho_{\gamma_k}$ in $C([0,T];L^1(\RR^d))$ follows from (\ref{eq:spatial_equicontinuity}).
\end{proof}

Now that we have established precompactness, in the rest of this subsection we will assume (without loss of generality) that $\gamma_k$ is a subsequence such that $(\rho_{1,\gamma_k}, \ldots, \rho_{\ell,\gamma_k}, p_{\gamma_k}, n_{\gamma_k})$ converges to a point $(\rho_{1}, \ldots, \rho_{\ell}, p, n)$ where the convergence holds in the spaces that we identified in Lemma \ref{lem:incompressible_precompact}.

We now establish some properties of the limit point.

\begin{lemma}
$(\rho_{1}, \ldots, \rho_{\ell}, p, n)$ is a weak solution to the incompressible system (\ref{eq:incompressible}-\ref{eq:incompressible_nutrients}).  Furthermore, the set $\{x\in \RR^d: \rho(t,x)>0, p(t,x)=0\}$ has measure zero for any $t\geq 0$, $\rho$ is nondecreasing in time almost everywhere, and  $p$ satisfies the complementarity formula
\begin{equation}\label{eq:variational_complementarity}
    p=\argmin_{\vp(1-\rho)=0} \int_{Q_T} \frac{1}{2}|\nabla \vp|^2-\vp G,
\end{equation}
where $G=\sum_{i=1}^{\ell}\frac{\rho_i}{\rho} G_i(p,n)$.
\end{lemma}
\begin{proof}
 The convergence properties that we have are strong enough to guarantee that $(\rho_1, \ldots, \rho_{\ell}, p, n)$ is a weak solution to the equations
 \[
 \partial_t \rho_i-\nabla\cdot (\rho_i\nabla p)=\rho_iG_i(p,n)
 \]
 \[
 \partial_t n-\alpha\Delta n=\-\sum_{i=1}^{\ell}\beta_i\rho_i.
 \]
 To prove that $(\rho_1, \ldots, \rho_{\ell}, p, n)$ is a solution to the incompressible system (\ref{eq:incompressible}-\ref{eq:incompressible_nutrients}) we still need to show that $\rho\leq 1$ and $p(1-\rho)=0$ almost everywhere.  Since $p_{\gamma_k}\leq p_h$ almost everywhere, it follows that $\rho_{\gamma_k}\leq p_h^{\frac{1}{\gamma_k}}$ almost everywhere.  Therefore, $\rho\leq 1$ almost everywhere. 
Fix some $\epsilon>0$ and some set $E\subset Q_T$ with finite measure.   We can then compute
\[
\int_{E} p(t,x)(1-\rho(t,x))\,dx\, dt=\lim_{k\to\infty} \int_{E} p_{\gamma_k}(t,x)(1-\rho_{\gamma_k}(t,x))
\]
\[
\leq \lim_{k\to\infty} \int_{E} \epsilon p_{\gamma_k}+(1-\epsilon)^{\gamma_k}\leq \epsilon\norm{p}_{L^1(E)},
\]
where the first inequality follows from splitting $E$ into the sets $\{(t,x)\in E: \rho_{\gamma_k}<1-\epsilon\}$ and $\{(t,x)\in E: \rho_{\gamma_k}\geq 1-\epsilon\}$.  Sending $\epsilon\to 0$ we can conclude that $p(1-\rho)=0$ almost everywhere.

Now that we know that $(\rho_1, \ldots, \rho_{\ell}, p, n)$ satisfies (\ref{eq:incompressible}-\ref{eq:incompressible_nutrients}) we can glean some more information.  Summing (\ref{eq:incompressible}) over the populations, we see that $\rho, p$ are weak solutions of the equation
\begin{equation}\label{eq:hele_shaw}
\partial_t \rho-\nabla \cdot(\rho\nabla p)=\rho G\, \quad p(1-\rho)=0, \quad \rho \leq 1,
\end{equation}
which is the Hele-Shaw equation with a source term.

By Proposition \ref{prop:u_l1}, we know that $\int_{\RR^d\times \{t\}} \rho\log(1+\frac{1}{p})^{\lambda'}$ is bounded for some $\lambda'>0$ and any $t\geq 0$.  Thus, $\{x\in \RR^d: \rho(t,x)>0, p(t,x)=0\}$ must have measure zero for all $t\geq 0$.

To see that $\rho$ is nondecreasing in time, we note that equation (\ref{eq:pme}) gives
\[
\partial_t \rho_{\gamma_k}+\frac{1}{\gamma_k}\rho_{\gamma_k}u_{\gamma_k}=\nabla \rho_{\gamma_k}\cdot\nabla p_{\gamma_k}.
\]
We then have the trivial inequalities
\[\frac{1}{\gamma_k}\rho_{\gamma_k}u_{\gamma_k}\leq \frac{1}{\gamma_k}\rho_{\gamma_k}u_{\gamma_k\, +}\leq  \frac{1}{2\gamma_k^{1/2}}\rho_{\gamma_k}+\frac{1}{2\gamma_k^{1/2}}(\frac{1}{\gamma_k}\rho_{\gamma_k}u_{\gamma_k\, +}^2).
\]
Recalling that $\omega(p_{\gamma_k})=\gi\rho_{\gamma_k}$ is a valid weight satisfying (W1-W3), we know that $\frac{1}{\gamma_k}\rho_{\gamma_k}u_{\gamma_k\, +}^2$ is uniformly bounded with respect to $\gamma_k$ in $L^1(Q_T)$.  Hence, for any nonnegative test function $\psi$, we have 
\[
\int_{Q_T} \psi\partial_t \rho=\lim_{k\to\infty}\int_{Q_T} \psi\partial_t \rho_{\gamma_k}=\lim_{k\to\infty}\int_{Q_T} \psi\nabla \rho_{\gamma_k}\cdot \nabla p_{\gamma_k}\geq 0.
\]
Thus, $\partial_t \rho\geq 0$ almost everywhere in space and time.

Finally, the complementarity condition (\ref{eq:variational_complementarity}) is a consequence of the weak equation (\ref{eq:hele_shaw}) when the pressure has $L^2([0,T];H^1(\RR^d))$ regularity
see for instance \cite{pqv, perthame_david, guillen_kim_mellet, jacobs_2021} (one can also derive the condition from our control on $u$).

\end{proof}

The incompressible analogues of Lemma \ref{lem:flow_cauchy} and Propositions \ref{prop:flow_uniqueness} and \ref{prop:continuity_uniqueness} now all follow from the same proofs used in Section \ref{ssec:compactness_gamma}

 Finally we can prove Theorem \ref{thm:main2}.

\begin{proof}[Proof of Theorem \ref{thm:main2}]
We have already established that $(\rho_1,\ldots, \rho_{\ell}, p, n)$ is a solution to the incompressible system (\ref{eq:incompressible}-\ref{eq:incompressible_nutrients}). The strong convergence of the $X_k, Y_k$ to $X, Y$ on the support of $\rho$ implies that $X$ and $Y$ satisfy all the properties in Definition \ref{def:complete_lagrangian_solution} when restricted to the support of $\rho$.  Thus, $(\rho_1,\ldots, \rho_{\ell}, p, n)$ is a complete Lagrangian solution to the incompressible system (\ref{eq:incompressible}-\ref{eq:incompressible_nutrients}).

The proof of the nonmixing property is identical to the proof of the nonmixing property in Theorem \ref{thm:main1}.

\end{proof}

\bibliographystyle{alpha}
\bibliography{references}
\end{document}